\documentclass[12pt]{amsart}
\usepackage[section]{placeins}
\usepackage{float}
\usepackage{tikz}
\usetikzlibrary{decorations.pathreplacing,calligraphy}
\usetikzlibrary{arrows,shapes,automata,backgrounds,petri,patterns}
\tikzset{
	 rootvertex/.style={fill=teal, rectangle,inner sep=2.5pt},	 
	 whiteroot/.style={fill=white, draw,rectangle,inner sep=2pt},	 
	 invertex/.style={fill=black,circle,inner sep=1pt},
	 oldroot/.style={fill=black,rectangle,inner sep=1.4pt},	 
	 rleaf/.style={fill=red,draw=red,circle,inner sep=1.5pt},	 
	 wleaf/.style={fill=white,draw,circle,inner sep=1.5pt},
	 centro/.style={fill=blue,draw,circle,inner sep=1.5pt},	 
	 pat1/.style={fill=purple!20},	 
	 pat2/.style={fill=blue!20}, 
	 pat3/.style={fill=gray!20}
	 }

\usepackage{todonotes}
\usepackage{graphicx}
\usepackage{psfrag}
\usepackage{epsf}
\usepackage{hyperref}
\usepackage[cp1250]{inputenc}
\usepackage{amsmath,amsfonts,amssymb,latexsym}
\usepackage[left=2cm,right=2cm,bottom=4cm]{geometry}
\usepackage{setspace}
\usepackage{enumitem}
\usepackage{multirow}
\usepackage{caption}

\newtheorem{theorem}{Theorem}

\newtheorem{corollary}[theorem]{Corollary}
\newtheorem{lemma}[theorem]{Lemma}

\newtheorem{problem}[theorem]{Problem}

\theoremstyle{definition}
\newtheorem{definition}[theorem]{Definition}

\newcommand{\T}{\mathcal{T}}
\usepackage{comment}

\def\vt{t\kern-0.22em\raise.18ex\hbox{\char'47}\lower.18ex\hbox{}\kern-0.08em}
\newcommand{\old}[1]{{}}
\newcommand{\whiteleaf}[1]{\left\langle #1 \right\rangle}
\newcommand{\mast}{\operatorname{mast}}

\usepackage{xcolor}

\title{Universal rooted phylogenetic tree shapes and universal tanglegrams }
\thanks{This material is based upon work supported by the National Science Foundation under Grant Number DMS 1641020.}
\author[A. Clifton]{Ann Clifton}
\author[\'E. Czabarka]{\'Eva Czabarka}
\author[K. Liu]{Kevin Liu} 
\author[S. Loeb]{Sarah Loeb}
\author[U. Okur]{Utku Okur}
\author[L. Sz\'ekely]{L\'aszl\'o Sz\'ekely}
\author[K. Wicke]{Kristina Wicke}
\address{Ann Clifton\\ Louisiana Tech University}
\email{aclifton@latech.edu}
\address{\'Eva Czabarka\\ University of South Carolina}
\email{czabarka@math.sc.edu}
\address{Kevin Liu\\University of Washington}
\email{kliu15@uw.edu}
\address{Sarah Loeb\\ Hampden-Sydney College}
\email{sloeb@hsc.edu}
\address{Utku Okur\\ University of South Carolina}
\email{uokur@email.sc.edu}
\address{L\'aszl\'o Sz\'ekely\\ University of South Carolina}
\email{szekely@math.sc.edu}
\address{Kristina Wicke\\  New Jersey Institute of Technology}
\email{kristina.wicke@njit.edu}


\begin{document} 
\begin{abstract}   
We provide an $\Omega(n\log n) $ lower bound and an $O(n^2)$ upper bound for the smallest size of   rooted binary trees (a.k.a. phylogenetic tree shapes), which are universal for rooted binary trees with $n$ leaves, i.e., contain all of them as induced binary subtrees. We explicitly compute the smallest universal trees for $n\leq 11$. 
We also provide  an $\Omega(n^2) $ lower bound and an $O(n^4)$ upper bound for the smallest  size of tanglegrams, which are universal for size $n$ tanglegrams, i.e., which contain all of them as induced subtanglegrams. Some of our results
generalize to rooted $d$-ary trees and to $d$-ary tanglegrams.
\end{abstract}

\maketitle

\section{Introduction}
A {\it rooted binary tree} with $n\geq 2$ leaves has $n$ vertices of degree one, $n-2$ vertices of degree three, and one vertex, the root, of degree two. A single vertex is considered
a rooted binary tree with $1$ leaf.  Vertices, other than the root, are not labeled, and the branches of the tree are not ordered.

These kinds of rooted binary trees are exactly the {\it rooted phylogenetic tree shapes} or {\it rooted tree topologies}, i.e., the unlabeled trees underlying the leaf-labeled phylogenetic trees; see \cite{SempleSteel}.  The number of  $n$-leaf rooted phylogenetic tree shapes is the Wedderburn-Etherington number $W_n$ given as sequence
A001190 in The On-Line Encyclopedia of Integer Sequences (OEIS) \cite{OEIS}. There is an asymptotic formula 
\begin{equation} \label{Wnasymp}
W_n\sim a \cdot \frac{b^n}{n^{3/2}},
\end{equation}
which goes back to Otter \cite{otter}; see also \cite{Landau1977}. Numerically, $b \approx 2.4832$ and $a\approx 0.3188$.  

A rooted binary tree $T_1$ is an {\it induced binary subtree} of a rooted binary tree $T_2$, if there is a subdivision $T_1'$ of $T_1$ such that  $T_1'$ is a subtree of $T_2$ and every leaf of $T_1'$ is a leaf of $T_2$. We are interested in the smallest number of leaves of a rooted binary tree, which we will denote by $u(n)$, that contains all $n$-leaf rooted binary trees as induced binary subtrees. We are also interested in characterizing the set of rooted binary trees with $u(n)$ leaves that contain all $n$-leaf rooted binary trees, for small values of $n$. 

This type of question is {\it universality} for all kinds of structures.
The universality problem for induced subgraphs was introduced by R. Rado \cite{rado}.
It has been studied and solved by M. Goldberg and E. Lifshitz for rooted trees \cite{goldberg}, by
F.R.K. Chung, R.L. Graham, and D. Coppersmith \cite{Coppersmith} for trees, and by F. R. K. Chung, R. L. Graham, and J. Shearer \cite{cater} for so-called caterpillars.
For universality of  permutations, see  the survey of M. Engen and  V. Vatter \cite{engen}. Motivated by universality of permutations, C. Defant, N. Kravitz, and A. Shah \cite{supertrees} investigated universality in plane trees, for several variants of the question. Some variants of their problems allowed subdivisions.

In case of rooted binary trees, it is easy to obtain an upper bound  $u(n)< n^{c \log n}$ from the obvious recursion
\begin{equation} \label{eq:direkt}
u(n)\leq u(n-1)+u\left(\left\lfloor \frac{n}{2}\right\rfloor\right).
\end{equation}
To justify Inequality~\eqref{eq:direkt}, consider any rooted binary tree with $n$ leaves that we want to embed. Since it is binary, it splits into two subtrees at the root, where one has at most $\lfloor \frac{n}{2}\rfloor$ leaves, and then the other has at most $n-1$ leaves.
  
As the universality functions in \cite{goldberg}, \cite{Coppersmith}, and \cite{supertrees} are superpolynomial, it would not have been surprising if $u(n)$ was superpolynomial, but in this paper we show that $u(n)=O(n^2)$.

Setting good lower bounds for universality problems is often hard. Trivially, $u(n)=\Omega(n)$, and we prove here that $u(n)=\Omega(n\log n)$.

Additionally, we made extensive calculations to compute $u(n)$ for small values of $n$. We list all $n$-universal trees for $n\leq 11$ in Section~\ref{sec:pics}.
To our surprise, the computed values of $u(n)$ coincided for a while with the terms of the sequence A173382 from \cite{OEIS}, and even when they differed, they remained close. Sequence A173382 is defined as the  partial sums of sequence A074206, which counts the ordered factorizations of the number $n$. We have no explanation for this phenomenon, which might be a pure coincidence. The table below compares the two sequences. The sequence A173382  was studied by L. Kalm\'ar \cite{kalmarmagyar}, \cite{kalmarnemet}, who showed 
that the sequence is asymptotically $c\cdot n^s$, where $s>1$ is the single real root of $\zeta(s)=2$. Numerically, $s\approx 1.72864723$.

\begin{table}[htbp]
\begin{tabular}{|c|c|c|c|c|c|c|c|c|c|c|c|c|}\hline
	& $ 1 $ & $2$ & $3$ & $4$ & $5$ & $6$ & $7$ & $8$ & $9$ & $10$ & $11$ & $12$\\\hline
A173382 & $ 1 $ & $2$ & $3$ & $5$ & $6$ & $9$ & $10$ & $14$ & $16$ & $19$ & $20$ & $28$  \\\hline
$u(n)$ & $ 1 $ & $2$ & $3$ & $5$ & $6$ & $9$ & $10$ & $14$ & $16$ & $19$ & $21$ & $\leq 28$ \\\hline
\end{tabular}
\caption{$u(n)$ versus the Kalm\'ar sequence A173382.}\label{tab:kalmar}
\end{table}

In addition to considering universality for rooted binary trees, we also consider universality for tanglegrams. 
While we give formal definitions below, a (binary) tanglegram is a graph consisting of two rooted binary trees on the same number of leaves together with a perfect matching between their leaf sets. Tanglegrams also appear in phylogenetics as a way to illustrate and compare the correspondence between two phylogenetic trees as well as to indicate co-evolution, e.g., in terms of hosts and parasites. 
In Section~\ref{sec:kovetk}, we obtain an $O(n^4)$ upper bound for the size of the smallest $n$-universal tanglegram as 
a corollary of our result for $n$-universal trees. We also set a $\Omega(n^2)$ lower bound for this quantity. 

Finally, we extend our results from binary to $d$-ary trees and tanglegrams. This is motivated by the fact that while leaf-labeled binary trees frequently are referred to as fully resolved phylogenetic trees, in practice, phylogenetic trees are not always fully resolved (e.g., due to conflict in the data), resulting in vertices with more than two children. Trees where each internal vertex has at most $d$ children are $d$-ary trees, and $d$-ary tanglegrams (see e.g., \cite{dtangle}) are formed from a pair of such trees. The notion of inducibility has been extended to $d$-ary trees
(see \cite{dary}), and consequently the universality problem arises in these objects as well.
Our upper bound arguments for binary trees easily extend to $d$-ary trees and $d$-ary tanglegrams, and we present these extensions in Section~\ref{sec:dary}.

\section{Definitions and preliminaries}

\subsection{The basics}

We start with some necessary definitions.

\begin{definition} A \emph{rooted tree} is a tree with a special vertex called the root; we will denote the root of $T$ by $r_T$. In a rooted tree $T$, vertex $x$ is an \emph{ancestor} of vertex $y$ (in other words, $y$ is a \emph{descendant} of $x$) if
$x$ lies on the unique $r_T$-$y$ path. A \emph{parent} is an ancestor that is also a neighbor, and a \emph{child} is a descendent that is also a neighbor. A \emph{leaf} is a vertex that has no children (note that if the rooted tree has only one vertex, 
it is both a root and a leaf).  If $x,y$ are vertices of a rooted tree $T$, then the last common ancestor of $x$ and $y$ is $z$ if the intersection of the $r_T$-$x$ and $_T$-$y$ paths is the $r_T$-$z$ path. Non-leaf vertices in a tree are called
\emph{internal vertices}.
\end{definition}

\begin{definition}
A rooted tree is \emph{binary} if every non-leaf vertex has exactly two children. A single vertex tree, where the 
vertex is considered both root and leaf, is also a binary tree.
\end{definition}

\begin{definition} A \emph{redleaf tree} is a tree such that one of its leaves is colored red. 
If $T$ is a redleaf tree, we denote the red leaf by $\ell_T$. For shortness, we refer to rooted binary redleaf trees
as \emph{r-$2$-r trees}.
\end{definition}

\begin{definition} Let $T$ be a tree (redleaf or not). The \emph{white leaves} of $T$ are all the leaves that are not red. $|T|$ will denote the number of vertices and
$\whiteleaf{T}$ will denote the number of white leaves of $T$.
\end{definition}

\begin{definition} Let $T$ be a rooted binary tree (possibly with a red leaf) and $v$ a vertex of $T$. The \emph{subtree rooted at $v$} is the subtree formed by $v$ and all of its descendants.
\end{definition}

Note that we may refer to the leaves of a rooted binary tree as white leaves.

In the illustrating figures, roots and leaves will be larger vertices. Leaves are white or red circles (depending on the color of the leaf), roots are green rectangles (if the root is also a leaf, its color depends on the leaf color), and internal non-root vertices are smaller or a different color.

\subsection{Tree operations and special tree families}

\begin{definition} Let $R$ and $Q$ be two rooted trees, at most one of which is a redleaf tree. Then, the tree $T=R\oplus Q$ is obtained by joining a root vertex $r_{R\oplus Q}$ to the roots of $R$ and $Q$, i.e., by introducing a new vertex $r_{R\oplus Q}$ and the edges $(r_{R\oplus Q},r_R)$ and $(r_{R\oplus Q},r_Q)$. $T$ is a redleaf tree precisely when one of $R$ and $Q$ is a redleaf tree (see {\rm Figure~\ref{fig:operations}}).
\end{definition}

Note that $\whiteleaf{R\oplus Q}=\whiteleaf{R}+\whiteleaf{Q}$.

\begin{definition} Given a rooted redleaf tree $R$ and another rooted (redleaf or not) tree $Q$, the tree $T=R[Q]$ is the rooted tree  obtained by identifying the root of $Q$ with $\ell_R$ and removing the coloring of $\ell_R$.
Here, $r_{R[Q]}=r_R$, and $R[Q]$ is a redleaf tree precisely when $Q$ is a redleaf tree (with $\ell_{R[Q]}=\ell_Q$) (see {\rm Figure~\ref{fig:operations}}).
\end{definition}

Note that $\whiteleaf{R[Q]}=\whiteleaf{R}+\whiteleaf{Q}$.

\begin{figure}[htb]
\centering
\begin{tikzpicture}[scale=.9,font=\tiny]
	       	\node at (0,4.8) {$r_R$};
	        \draw (.25,4.25)--(0.5,4);     	        
	        \draw (-.5,4)--(0,4.5)--(.25,4.25)--(0,4);			        
	        \node[wleaf]  at (-.5,4) {}; 
	        \node[invertex]  at (.25,4.25) {}; 
	        \node[wleaf]  at (0,4) {}; 
	        \node[rleaf]  at (.5,4) {}; 
	        \node[rootvertex]  at (0,4.5) {}; 
	        \node at (.5,3.7) {$\ell_R$};
	        \node at (0,3.2) {The tree $R$};
	       	 \node at (3,5.05) {$r_Q$};
		 \draw (3.75,4)--(3.5,4.25)--(3.25,4);
		 \draw (2.75,4)--(2.5,4.25)--(2.25,4);
		 \draw (2.5,4.25)--(3,4.75)--(3.5,4.25);		        
		 \node[rootvertex]  at (3,4.75) {}; 
	        \node[invertex]  at (2.5,4.25) {}; 
	        \node[invertex]  at (3.5,4.25) {}; 
	        \node[wleaf]  at (2.25,4) {}; 
		 \node[wleaf]  at (2.75,4) {}; 
		 \node[wleaf]  at (3.25,4) {}; 
		 \node[wleaf]  at (3.75,4) {}; 
       		        \node at (3,3.2) {The tree $Q$};
	        \node at (7,5.8) {$r_{R\oplus Q}$}; 
	        \draw (6.25,4.25)--(6.5,4);    
	        \draw (5.5,4)--(6,4.5)--(6.25,4.25)--(6,4);	
		 \draw (8,4)--(8.25,4.25)--(8.5,4);
		 \draw (7,4)--(7.25,4.25)--(7.5,4);
		 \draw (7.25,4.25)--(7.75,4.75)--(8.25,4.25);
		   \draw (7.75,4.75)--(7,5.5)--(6,4.5);  	                 	        		   	  		 	   		        
	        \node[wleaf]  at (5.5,4) {}; 
	        \node[invertex]  at (6.25,4.25) {}; 
	        \node[wleaf]  at (6,4) {}; 
	        \node[rleaf]  at (6.5,4) {}; 
	        \node[oldroot]  at (6,4.5) {}; 
	        \node at (6.5,3.7) {$\ell_{R\oplus Q}$};
	        \node[oldroot]  at (7.75,4.75) {}; 
	        \node[invertex]  at (7.25,4.25) {}; 
	        \node[invertex]  at (8.25,4.25) {}; 
	        \node[wleaf]  at (7,4) {}; 
		 \node[wleaf]  at (7.5,4) {}; 
		 \node[wleaf]  at (8,4) {}; 
		 \node[wleaf]  at (8.5,4) {}; 
		   \node[rootvertex]  at (7,5.5) {}; 
       		        \node at (7,3.2) {The tree $R\oplus Q$};
	        \node at (5.7,4.6) {$r_{R}$}; 
	        \node at (8.1,5.1) {$r_{Q}$}; 
		 \node at (10.5,5.6) {$r_{R[Q]}=r_R$};	
	        \draw (10.75,5)--(11,4.75);     	        
	        \draw (10,4.75)--(10.5,5.25)--(10.75,5)--(10.5,4.75);
		 \draw (11.75,4)--(11.5,4.25)--(11.25,4);
		 \draw (10.75,4)--(10.5,4.25)--(10.25,4);
		 \draw (10.5,4.25)--(11,4.75)--(11.5,4.25);		         
	        \node[wleaf]  at (10,4.75) {}; 
	        \node[invertex]  at (10.75,5) {}; 
	        \node[wleaf]  at (10.5,4.75) {}; 
	        \node[oldroot]  at (11,4.75) {}; 
	        \node[rootvertex]  at (10.5,5.25) {}; 
	        \node[invertex]  at (10.5,4.25) {}; 
	        \node[invertex]  at (11.5,4.25) {}; 
	        \node[wleaf]  at (10.25,4) {}; 
		 \node[wleaf]  at (10.75,4) {}; 
		 \node[wleaf]  at (11.25,4) {}; 
		 \node[wleaf]  at (11.75,4) {}; 
		   \node at (10.5,3.2) {The tree $R[Q]$};  	  		 	   
	        \node at (11.3,4.8) {$r_{Q}$}; 
\end{tikzpicture}
\caption{$R\oplus Q$ and $R[Q]$.}\label{fig:operations}
\end{figure}
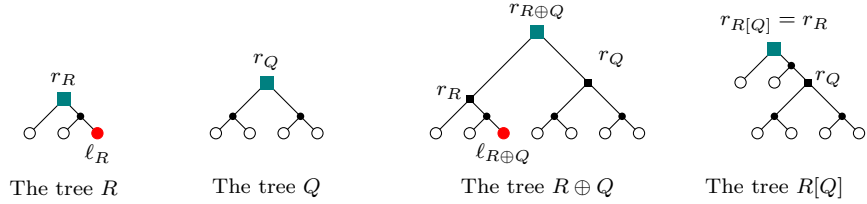

We define here some tree families, which will be needed  later. 
\begin{definition} The $n$-leaf \emph{rooted caterpillar} $C_n$ is defined recursively as follows: $C_1$ is the unique $1$-leaf rooted binary tree, and for $n\ge 2$, we have $C_n=C_1\oplus C_{n-1}$ (see \rm{Figure~\ref{fig:cater}}).
\end{definition}

\begin{figure}[htb]
\centering
\begin{tikzpicture}[scale=.9,font=\tiny]
	        \node[whiteroot]  at (0,0) {}; 
	        \node at (0,-.5) {$C_1$};
	        \end{tikzpicture}
\quad
\begin{tikzpicture}[scale=.9,font=\tiny]
	        \draw (-.25,-.25)--(0,0)--(0.25,-.25);
	        \node[wleaf]  at (-.25,-.25) {}; 
	        \node[wleaf]  at (.25,-.25) {}; 
	        \node[rootvertex]  at (0,0) {}; 
	        \node at (0,-.75) {$C_2$};
	        \end{tikzpicture}
\quad
\begin{tikzpicture}[scale=.9,font=\tiny]
	        \draw (0,-.5)--(.25,-.25)--(.5,-.5);
	        \draw (-.25,-.25)--(0,0)--(0.25,-.25);
	        \node[wleaf]  at (.5,-.5) {}; 
	        \node[wleaf]  at (0,-.5) {}; 
	        \node[invertex]  at (-.25,-.25) {}; 
	        \node[invertex]  at (.25,-.25) {}; 
	        \node[rootvertex]  at (0,0) {}; 
	        \node at (0.125,-1) {$C_3$};
	        \end{tikzpicture}
\quad
\begin{tikzpicture}[scale=.9,font=\tiny]
	        \draw (.25,-.75)--(.5,-.5)--(.75,-.75);
	        \draw (0,-.5)--(.25,-.25)--(.5,-.5);
	        \draw (-.25,-.25)--(0,0)--(0.25,-.25);
	        \node[wleaf]  at (.25,-.75) {}; 
	        \node[wleaf]  at (.75,-.75) {}; 
	        \node[invertex]  at (.5,-.5) {}; 
	        \node[wleaf]  at (0,-.5) {}; 
	        \node[wleaf]  at (-.25,-.25) {}; 
	        \node[invertex]  at (.25,-.25) {}; 
	        \node[rootvertex]  at (0,0) {}; 
	        \node at (0.25,-1.25) {$C_4$};
	        \end{tikzpicture}
\quad
\begin{tikzpicture}[scale=.9,font=\tiny]
	        \draw (1,-1)--(.75,-.75)--(.5,-1);
	        \draw (.25,-.75)--(.5,-.5)--(.75,-.75);
	        \draw (0,-.5)--(.25,-.25)--(.5,-.5);
	        \draw (-.25,-.25)--(0,0)--(0.25,-.25);
	        \node[wleaf]  at (1,-1) {}; 
	        \node[wleaf]  at (.5,-1) {}; 
	        \node[wleaf]  at (.25,-.75) {}; 
	        \node[invertex]  at (.75,-.75) {}; 
	        \node[invertex]  at (.5,-.5) {}; 
	        \node[wleaf]  at (0,-.5) {}; 
	        \node[wleaf]  at (-.25,-.25) {}; 
	        \node[invertex]  at (.25,-.25) {}; 
	        \node[rootvertex]  at (0,0) {}; 
	        \node at (0.45,-1.5) {$C_5$};
	        \end{tikzpicture}
\caption{The rooted caterpillars $C_n$ for $n\le 5$.}\label{fig:cater}
\end{figure}
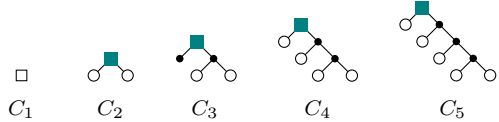

\begin{definition}
The \emph{height} of a  rooted binary tree is the length of a longest path from the root to a leaf.
\end{definition}
For example, the height of $C_n$ is $n-1$.

\begin{definition}
Let $h\in\mathbb{N}$. The \emph{complete binary tree $B_h$ of height $h$} is defined recursively as follows: $B_0$ is the unique 1-leaf rooted binary tree, and for $h\ge 1$, we have $B_h=B_{h-1}\oplus B_{h-1}$. Clearly, $\whiteleaf{B_h}=2^h$ (see \rm{Figure~\ref{fig:complete}}).
\end{definition}

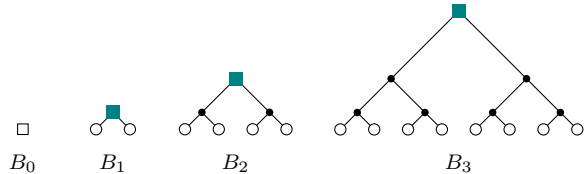
\begin{figure}[htb]
\centering
\begin{tikzpicture}[scale=.9,font=\tiny]
	        \node[whiteroot]  at (0,0) {}; 
	        \node at (0,-.5) {$B_0$};
	        \end{tikzpicture}
\quad
\begin{tikzpicture}[scale=.9,font=\tiny]
	        \draw (-.25,-.25)--(0,0)--(0.25,-.25);
	        \node[wleaf]  at (-.25,-.25) {}; 
	        \node[wleaf]  at (.25,-.25) {}; 
	        \node[rootvertex]  at (0,0) {}; 
	        \node at (0,-.75) {$B_1$};
	        \end{tikzpicture}
\quad
\begin{tikzpicture}[scale=.9,font=\tiny]
	        \draw (0,-.25)--(.25,0)--(.5,-.25);
	        \draw (1,-.25)--(1.25,0)--(1.5,-.25);	        
	        \draw (.25,0)--(.75,.5)--(1.25,0);
	        \node[wleaf]  at (.5,-.25) {}; 
	        \node[wleaf]  at (0,-.25) {}; 
	        \node[wleaf]  at (1,-.25) {}; 
	        \node[wleaf]  at (1.5,-.25) {}; 
	        \node[invertex]  at (1.25,0) {}; 
	        \node[invertex]  at (.25,0) {}; 
	        \node[rootvertex]  at (.75,.5) {}; 
	        \node at (0.75,-.75) {$B_2$};
	        \end{tikzpicture}
\quad
\begin{tikzpicture}[scale=.9,font=\tiny]
	        \draw (0,-1)--(.25,-.75)--(.5,-1);
	        \draw (1,-1)--(1.25,-.75)--(1.5,-1);	        
	        \draw (.25,-.75)--(.75,-.25)--(1.25,-.75);
	        \draw (2,-1)--(2.25,-.75)--(2.5,-1);
	        \draw (3,-1)--(3.25,-.75)--(3.5,-1);	        
	        \draw (2.25,-.75)--(2.75,-.25)--(3.25,-.75);	        
	        \draw (2.75,-.25)--(1.75,.75)--(.75,-.25);	        	        	        
	        \node[wleaf]  at (.5,-1) {}; 
	        \node[wleaf]  at (0,-1) {}; 
	        \node[wleaf]  at (1,-1) {}; 
	        \node[wleaf]  at (1.5,-1) {}; 
	        \node[invertex]  at (1.25,-.75) {}; 
	        \node[invertex]  at (.25,-.75) {}; 
	        \node[invertex]  at (.75,-.25) {}; 
	        \node[wleaf]  at (2.5,-1) {}; 
	        \node[wleaf]  at (2,-1) {}; 
	        \node[wleaf]  at (3,-1) {}; 
	        \node[wleaf]  at (3.5,-1) {}; 
	        \node[invertex]  at (3.25,-.75) {}; 
	        \node[invertex]  at (2.25,-.75) {}; 
	        \node[invertex]  at (2.75,-.25) {}; 
	        \node[rootvertex]  at (1.75,.75) {}; 
	        \node at (1.75,-1.5) {$B_3$};
	        \end{tikzpicture}
\caption{The rooted complete binary trees $B_h$ for $h\le 3$.}\label{fig:complete}
\end{figure}

\begin{definition} For $h,\ell\in\mathbb{N}$, where $\ell\ge 2$, we define the \emph{$(h,\ell)$-jellyfish $J_{h,\ell}$} as follows: Take a rooted complete binary tree of height $h$ and identify each of its leaves with the root of a caterpillar $C_{\ell}$.
\end{definition}

Clearly $\whiteleaf{J_{h,\ell}}=2^h\ell$, the height of $J_{h,\ell}$ is $h+\ell-1$,
and $J_{h,2}=B_{h+1}$.

\subsection{Inducibility and universality}
\begin{definition} Let $T$ be a rooted binary tree (redleaf or not) and $S$ be a subset of its leaves. The \emph{binary subtree induced by $S$} is the rooted binary tree $T^{\star}$ obtained as follows:
Let $F$ be the smallest subtree  of $T$ that contains $S$, and let $r_{T^{\star}}$ be the vertex of $F$ closest to $r_T$. $F$ is a subdivision of a unique rooted binary tree, which is by definition  $T^{\star}$. In other words, we obtain $T^{\star}$ by suppressing the non-root  vertices  of degree 2 in $F$. The leaves maintain their original coloration, i.e., $T^{\star}$ is a redleaf tree precisely when $T$ is a redleaf tree and $S$ contains $\ell_T$. The tree $T^{\prime}$ is an \emph{induced binary subtree} of $T$ if it is induced by some subset $S^{\prime}$ of the leaves of $T$ (see \rm{Figure~\ref{fig:induce}}).
\end{definition}

\begin{figure}[htbp]
\centering
\begin{tikzpicture}[scale=1,font=\tiny]
	\draw (0,0)--(2.25,2.25)--(4.5,0);
	\draw (1.5,0)--(.75,.75);
	\draw (2,0)--(3.25,1.25);
	\draw (3,0)--(3.75,0.75);
	\draw (.5,0)--(0.25,0.25);
	\draw (1,0)--(1.25,0.25);
	\draw (2.5,0)--(2.25,0.25);
	\draw (3.5,0)--(3.25,0.25);
	\draw (4,0)--(4.25,0.25);
	\draw (2.25,2.25)--(2.5,2.5)--(5,0);
        \node[wleaf]  at (0,0) {};
        \node[wleaf]  at (1,0) {};
        \node[wleaf]  at (1.5,0) {};
        \node[wleaf]  at (2,0) {};
        \node[wleaf]  at (2.5,0) {};
        \node[wleaf]  at (3,0) {};
        \node[wleaf]  at (3.5,0) {};
        \node[wleaf]  at (4,0) {};
        \node[wleaf]  at (4.5,0) {};
        \node[wleaf]  at (5,0) {};
        \node[rootvertex]  at (2.5,2.5) {};
        \node[invertex]  at (2.25,2.25) {};        
        \node[invertex]  at (0.75,0.75) {};
        \node[invertex]  at (3.25,1.25) {};
        \node[invertex]  at (3.75,0.75) {};
        \node[invertex]  at (0.25,0.25) {};
        \node[invertex]  at (1.25,0.25) {};
        \node[invertex]  at (2.25,0.25) {};
        \node[invertex]  at (3.25,0.25) {};
        \node[invertex]  at (4.25,0.25) {};

	\node at (.5,-0.25) {$x_1$};
	\node at (1.5,-0.25) {$x_2$};
	\node at (2,-0.25) {$x_3$};
	\node at (4,-0.25) {$x_4$};
        \node[rleaf]  at (.5,0) {};
	
\end{tikzpicture}
\quad\quad\quad
\begin{tikzpicture}[scale=1,font=\tiny]
	\draw (16,0)--(16.75,0.75)--(17.5,0);
	\draw (16.25,0.25)--(16.5,0);
	\draw (17.25,0.25)--(17,0);
        \node[wleaf]  at (16.5,0) {};
        \node[wleaf]  at (17,0) {};
        \node[wleaf]  at (17.5,0) {};
        \node[invertex]  at (16.25,0.25) {};
        \node[rootvertex]  at (16.75,0.75) {};
        \node[invertex]  at (17.25,0.25) {};
	\node at (16,-0.25) {$x_1$};
	\node at (16.5,-0.25) {$x_2$};
	\node at (17,-0.25) {$x_3$};
	\node at (17.5,-0.25) {$x_4$};
        \node[rleaf]  at (16,0) {};
\end{tikzpicture}
\caption{A rooted binary redleaf tree with $4$ selected leaves marked $x_1,x_2,x_3,x_4$ and the tree induced by the selected leaves.}\label{fig:induce}
\end{figure}
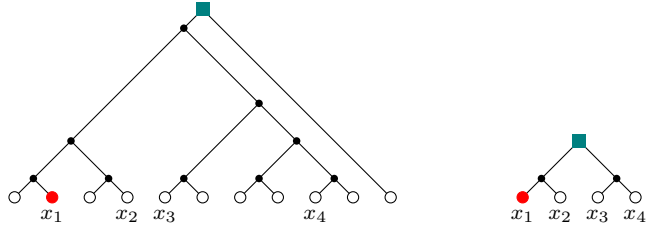

\begin{definition} A rooted binary tree is \emph{$n$-universal} if it contains all $n$-leaf rooted binary trees as induced binary subtrees. $\mathcal{U}(n)$ denotes the class of $n$-universal trees and $u(n)=\min\{\whiteleaf{T}:T\in\mathcal{U}(n)\}$.
An r-$2$-r tree is \emph{$n$-universal} if it contains all r-$2$-r trees $T$ with $n$ white leaves as induced binary subtrees. $\mathcal{U}_1(n)$ denotes the class of $n$-universal r-$2$-r trees and $u_1(n)=\min\{\whiteleaf{T}:T\in\mathcal{U}_1(n)\}$.
\end{definition}

\subsection{Maximum agreement subtree}

\begin{definition} Let $T_1,T_2$ be rooted binary trees. The rooted binary tree $R$ is an \emph{agreement subtree of $T_1$ and $T_2$}, if
it is an induced binary subtree of both $T_1$ and $T_2$.  $R$ is a \emph{maximum agreement subtree of $T_1$ and $T_2$}, if it is an agreement subtree of them and
$\whiteleaf{R}$ is maximal amongst all agreement subtrees. We will denote by $\mast(T_1,T_2)$  the number of leaves in a maximum agreement subtree.
\end{definition}
Note that the concept of maximum agreement subtrees has been much investigated in the labelled phylogenetic tree context, e.g.,~\cite{bordewich,markin,masttree}. In the following, we show that the maximum agreement subtree of two jellyfishes is a jellyfish with the maximum possible height of the complete binary tree subgraph and the entire jellyfish, which will be useful in some lower bound arguments.

\begin{lemma}\label{lm:jelly} Assume that $h_1\le h_2$ and  $\ell_1,\ell_2\ge 2$, and set $m=\min(h_1+\ell_1-1,h_2+\ell_2-1)$.
Then $\mast (J_{h_1,\ell_1},J_{h_2,\ell_2})=\whiteleaf{J_{h_1,m-h_1+1}}=2^{h_1}(m+1-h_1)$.
\end{lemma}

\begin{proof}
Any common induced binary subtree of $J_{h_1,\ell_1}$ and $J_{h_2,\ell_2}$
has height at most $m$. Trivially, $J_{h_1,m-h_1+1}$ is a common induced binary subtree, so 
$\mast (J_{h_1,\ell_1},J_{h_2,\ell_2})\ge \whiteleaf{J_{h_1,m-h_1+1}}$.

Now consider a common induced binary subtree $T$ of $J_{h_1,\ell_1}$ and $J_{h_2,\ell_2}$. Set $T_0$ to be the subtree (in the usual sense) of $T$
containing the  vertices of $T$ at distance at most $h_1$ from $r_T$, and let $v_1,\ldots,v_s$ ($s\le 2^{h_1}$)
be the set of leaves of $T_0$ at distance $h_1$ from $r_T$
that are not leaves of $T$.
Then there are rooted binary trees $T_1,\ldots,T_s$ such that $T$ is obtained from $T_0$ by identifying each $v_i$ with the root of $T_i$. 
Moreover, as $T$ has height at most $m$ and $T$ is an induced binary subtree
of $J_{h_1,\ell_1}$, for each $1\le i\le s$ there is a $k_i:2\le k_i\le m-h_1+1$ such that $T_i=C_{k_i}$. 
But this means that $T$ is a subtree
of $J_{h_1,m-h_1+1}$. Therefore, $\whiteleaf{T}\le \whiteleaf{J_{h_1,m-h_1+1}}$, and consequently, 
$\mast (J_{h_1,\ell_1},J_{h_2,\ell_2})\le \whiteleaf{J_{h_1,m-h_1+1}}$.
\end{proof}

\subsection{Leaf centroid vertex}
We will need the following slightly modified version of the centroid.

\begin{definition} Let $T$ be any tree and $W$ be a class of vertices (e.g., leaf, internal, etc.).
For any vertex $v$ of $T$, 
$$\tau_W(v)=\max\{|W\cap (S\cup\{v\})|: S\text{ is a component of } T-v\}.$$
The vertex $w$ is a \emph{$W$-centroid vertex}, if $\tau_W(w)$ is minimal over all vertices of $T$.
The \emph{$W$-centroid} is the set of $W$-centroid vertices.
\end{definition}

We will use in particular white leaf centroid vertices, so we choose $W$ be the class of white leaves (see Figure~\ref{fig:centroid}).

\begin{figure}
\begin{tikzpicture}[scale=.7,font=\tiny]
  \node[wleaf] (l0) at (-.5,3) {};
  \node[wleaf] (l7) at (1.5,3) {};
	 \node[wleaf] (l1) at (0,3) {};
	 \node[wleaf] (l2) at (0.5,3) {};
	 \node[wleaf] (l3) at (1,3) {};
	 \node[wleaf] (l4) at (0,0) {};
	 \node[wleaf] (l5) at (0.5,0) {};
	 \node[wleaf] (l6) at (1,0) {};  
	 \node[invertex] (c1)at (0.5,0.5) {};
	 \node[invertex] (c2) at (1,.625) {};
	 \node[invertex] (c3) at (1.5,.75) {};
         \node at (1.2,1.5) {$x$};
	 \node[invertex] (c4) at (1,.875) {};
	 \node[invertex] (c5)at (0.5,1) {};
	 \node[invertex] (c6) at (0,1.125) {};
	 \node[invertex] (c7) at (-0.5,1.25) {};
	 \node[invertex] (c8) at (0,1.375) {};
	 \node[invertex] (c9)at (0.5,1.5) {};
	 \node[invertex] (c10) at (1,1.625) {};
	 \node[invertex] (c11) at (1.5,1.75) {};
	 \node[invertex] (c12) at (1,1.875) {};
	 \node[invertex] (c13)at (0.5,2) {};
	 \node[invertex] (c14) at (0,2.125) {};
	 \node[invertex] (c15) at (-0.5,2.25) {};
	 \node[invertex] (c16) at (0,2.375) {};
	 
	 \node[centro] (c17) at (0.5,2.5) {};	
	 \draw (l5)--(c1)--(c2);
  \draw (c2)--(c3)--(c4)--(c5)--(c6);
  \draw[densely dotted, thick] (c6)--(c7)--(c8)--(c9)--(c10)--(c11)--(c12)--(c13)--(c14);
  \draw (c14)--(c15)--(c16)--(c17);
  \draw (c17)--(l1);
  \draw (l0)--(c17)--(l7);
	 \draw (l4)--(c1)--(l6);
	 \draw (l2)--(c17)--(l3);	  
\end{tikzpicture}
\quad\quad\quad
\begin{tikzpicture}[scale=.7,font=\tiny]
	 \node[wleaf] (l1) at (0,3) {};
	 \node[wleaf] (l2) at (0.5,3) {};
	 \node[wleaf] (l3) at (1,3) {};
	 \node[wleaf] (l4) at (0,0) {};
	 \node[wleaf] (l5) at (0.5,0) {};
	 \node[wleaf] (l6) at (1,0) {};  
	 \node[centro] (d1)at (0.5,0.5) {};
	 \node[centro] (d2) at (0.5,1) {};
	 \node[centro] (d3) at (0.5,1.5) {};
  \node at (.2,1.5) {$x$};
	 \node[centro] (d4) at (0.5,2) {};
	 \node[centro] (d5) at (0.5,2.5) {};	
	 \draw (l5)--(d1);
	 \draw (l4)--(d1)--(l6);
	 \draw (l2)--(d5)--(l3);	 
	 \draw[densely dotted,thick] (d1)--(d2)--(d3)--(d4)--(d5);
	 \draw (d5)--(l1);	  
\end{tikzpicture}
\quad\quad\quad
\begin{tikzpicture}[scale=.7,font=\tiny]
	 \node[rleaf] (l1) at (0,3) {};
	 \node[wleaf] (l2) at (0.5,3) {};
	 \node[wleaf] (l3) at (1,3) {};
	 \node[wleaf] (l4) at (0,0) {};
	 \node[wleaf] (l5) at (0.5,0) {};
	 \node[wleaf] (l6) at (1,0) {};  
	 \node[centro] (c1)at (0.5,0.5) {};
	 \node[invertex] (c2) at (0.5,1) {};
	 \node[invertex] (c3) at (0.5,1.5) {};
	 \node[invertex] (c4) at (0.5,2) {};
	 \node[invertex] (c5) at (0.5,2.5) {};	 
	 \draw (l5)--(c1)--(c2)--(c3)--(c4)--(c5)--(l1);
	 \draw (l4)--(c1)--(l6);
	 \draw (l2)--(c5)--(l3);	 
\end{tikzpicture}
\quad\quad\quad
\begin{tikzpicture}[scale=.7,font=\tiny]
	 \node[wleaf] (l1) at (0,0) {};
	 \node[wleaf] (l2) at (0.5,0) {};
	 \node[wleaf] (l3) at (1,0) {};
	 \node[wleaf] (l4) at (1.5,0) {};
	 \node[wleaf] (l5) at (1.25,1.5) {};
	 \node[wleaf] (l6) at (1.75,1.5) {};
	 \node[wleaf] (l7) at (1.75,2.5) {};
	 \node[rootvertex] (root) at (1.5,3) {};
	 \node[invertex] (i1) at (0.25,0.5) {};
	 \node[invertex] (i2) at (1.25,0.5) {};
	 \node[centro] (i3) at (0.75,1.5) {};
  \node at (.45,1.5) {$x$};  
	 \node[invertex] (i4) at (1.5,2) {};
	 \node[invertex] (i5) at (1.25,2.5) {};	 
	\draw (l1)--(i1)--(l2);
	\draw (l3)--(i2)--(l4);
	\draw (i1)--(i3)--(i2);
	\draw (l6)--(i4)--(l5);	 
	\draw[densely dotted,thick] (i3)--(i5);
    \draw (i5)--(i4);
	\draw (l7)--(root)--(i5);
\end{tikzpicture}
\quad\quad\quad
\begin{tikzpicture}[scale=.7,font=\tiny]
	 \node[wleaf] (l1) at (0,0) {};
	 \node[wleaf] (l2) at (0.5,0) {};
	 \node[rleaf] (l3) at (1,0) {};
	 \node[wleaf] (l4) at (1.5,0) {};
	 \node[wleaf] (l5) at (1.25,1.5) {};
	 \node[wleaf] (l6) at (1.75,1.5) {};
	 \node[wleaf] (l7) at (1.75,2.5) {};
	 \node[rootvertex] (root) at (1.5,3) {};
	 \node[invertex] (i1) at (0.25,0.5) {};
	 \node[invertex] (i2) at (1.25,0.5) {};
	 \node[centro] (i3) at (0.75,1.5) {};
	 \node[invertex] (i4) at (1.5,2) {};
	 \node[centro] (i5) at (1.25,2.5) {};
	 
	\draw (l1)--(i1)--(l2);
	\draw (l3)--(i2)--(l4);
	\draw (i1)--(i3)--(i2);
	\draw (l6)--(i4)--(l5);	 
	\draw (i3)--(i5)--(i4);
	\draw (l7)--(root)--(i5);

\end{tikzpicture}
\caption{White leaf centroids in different (unrooted, unrooted, unrooted redleaf, rooted binary and r-$2$-r) trees. 
White leaf centroid vertices are colored blue. For the white leaf trees, the (usual) centroid vertices are marked with an $x$, and the $1/3$-$2/3$ separator edges (in terms of the number of vertices) are drawn as dotted lines.}
\label{fig:centroid}
\end{figure}
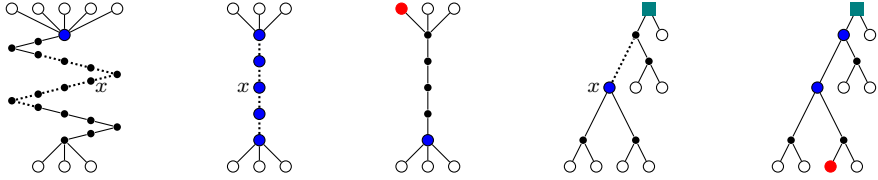

\begin{lemma}\label{lm:balancer} Let $T$ be any tree (rooted, unrooted, redleaf or not), and set $W$ to be the class of white leaves. 
If $T$ has an internal vertex and at least $2$ white leaves, then any
white leaf centroid vertex of $T$ is an internal vertex and $\tau_W(w)\le\left\lfloor\frac{\whiteleaf{T}}{2}\right\rfloor$.
\end{lemma}

\begin{proof} 
Let $T$ be a tree satisfying the assumptions, $W$ be the set of white leaves, and $n=|W|=\whiteleaf{T}$.
It is clear from the definition that for any leaf $v$ of $T$, $\tau_W(v)=n$. 
Let $v_1,v_2$ be two white leaves of $T$. Since $T$ has an internal vertex, the $v_1$-$v_2$ path in $T$ has at least one internal vertex $u$, and $\tau_W(u)\le n-1$. Therefore, the white leaf centroid of $T$ is an internal vertex. In addition, when $n=2$, then $n-1=\lfloor\frac{n}{2}\rfloor$, and the
theorem holds.

Now for contradiction, assume that $n\ge 3$, and let $u_1$ be an internal vertex with $\tau_W(u_1)\ge\lfloor\frac{n}{2}\rfloor+1\ge 2$ but $\tau_W(u_1)<n$. Note that for any vertex $x$, if $\tau_W(x)>\frac{n}{2}$, there is a unique component $T^{\prime}$ of $T-x$ with $|T^{\prime}\cap W|=\tau_W(x)$. We will define a sequence $u_1,\ldots,u_{k}$  of 
vertices
such that $u_1u_2\ldots u_k$ is a path, $\tau_W(u_1)=\tau_W(u_2)=\cdots=\tau_W(u_{k-1})>\tau_W(u_k)$, and for $i:2\le i\le k$, $u_i$ 
is in the (unique) component $T^{\star}$ of $T-u_{i-1}$ with $|T^{\star}\cap W|=\tau_W(u_1)$. Note that $u_2,\ldots,u_k$ must be a sequence
of internal vertices, as for any leaf $y$ we have $\tau_W(y)=n>\tau_W(u_1)$.
Since $\tau_W(u_k)<\tau_W(u_1)$, $u_1$ is not a white leaf centroid vertex, and this finishes the proof.

Assume that $u_1\ldots u_i$ has been defined for some $i\ge 1$ such that $\tau_W(u_1)=\tau_W(u_2)=\cdots=\tau_W(u_{i})$ and for $i:2\le i\le k$, $u_i$ 
is in the (unique) component $T^{\prime}$ of $T-u_{i-1}$ with $|T^{\prime}\cap W|=\tau_W(u_1)$. 
This means that there is a (unique) component $T^{\star}$ of $T-u_i$ with $|T^{\star}\cap W|=\tau_W(u_1)$, and $T^{\star}\cap \{u_1,\ldots,u_i\}=\emptyset$.
Let $x$ be the vertex of $T^{\star}$ that is adjacent to $u_i$.
As $|T^{\star}\cap W|=\tau_W(u_1)\ge 2$, $x$ is not a leaf of $T$. 
The components of $T-x$ are $T-T^{\star}$ with $|(T-T^{\star})\cap W|=n-\tau_W(u_1)<\tau_W(u_1)$ and
the components of $T^{\star}-x$; therefore $\tau_W(x)\le\tau_W(u_1)$ and $x$ is not on the $u_1\ldots u_i$ path.  
Set $x=u_{i+1}$. If $\tau_W(u_{i+1})<\tau_W(u_1)$, then $k=i+1$ and we can stop. Otherwise, the component 
$T'$ of $T-u_{i+1}$ with
$|T'\cap W|=\tau_W(u_1)$ is a subgraph of $T^{\star}-u_{i+1}$, and consequently $T'\cap\{u_1,\ldots,u_{i+1}\}=\emptyset$. As the sequence $u_1,\ldots,u_{i+1}$ satisfies the conditions, we can repeat the procedure.

The procedure must eventually stop, as the path can not be extended infinitely. The path can not reach a leaf vertex, as it does not contain vertices with $\tau_W$ value greater than $\tau_W(u_1)$. Therefore the procedure stops at an internal vertex with $\tau_W$ value less than $\tau_W(u_1)$.
\end{proof}

It is easy to see that for rooted binary trees a centroid vertex is also a white leaf centroid vertex. For trees in general, this is not the case, so this definition is needed for the results on $d$-ary trees.
Also, for any $n\geq 2$, and any rooted binary tree $T$ with
$n$ leaves, there exists an edge $e$ of $T$ incident with a white leaf centroid, such that the removal of $e$ partitions the $n$ leaves into two classes, such that the number of leaves in both classes is at most $2n/3$, or equivalently, such that the number of leaves in both classes is at least $n/3$. 

The existence of this separator edge was first discovered by P. Erd\H os, R.J. Faudree, C.C. Rousseau, and R.H. Schelp
\cite{EFRS}; and was rediscovered by P. L. Erd\H os,  M.A. Steel, L.A. Sz\'ekely, and T.J. Warnow \cite{artificial}. 

Note that the notion of separator edges can be extended in terms of  partitions of the set of all vertices (rather than just the leaves) into two classes in an analogous way (see Figure~\ref{fig:centroid}).

\section{Bounds on the universal tree size}

\subsection{A polynomial upper bound}
In this subsection, we establish a polynomial upper bound on $u(n)$. The proof of this bound depends on recursive upper bounds connecting $u(n)$ and $u_1(n)$, which will be established in Lemmas~\ref{lm:unhalf} and \ref{lm:u1nhalf}. 

\begin{lemma}\label{lm:unhalf} Let $n\ge 2$, 
$T_1\in\mathcal{U}_1\left(\left\lfloor\frac{n}{2}\right\rfloor\right)$ 
and $T_2, T_3\in\mathcal{U}\left(\left\lfloor\frac{n}{2}\right\rfloor\right)$. Then, $T_1[T_2\oplus T_3]\in\mathcal{U}(n)$ and consequently,
$$u(n)\leq u_1(\lfloor n/2 \rfloor) +2u(\lfloor n/2 \rfloor) .$$
\end{lemma}
\begin{figure}[htb]
\centering
\begin{tikzpicture}[scale=.8,font=\tiny]
	        \draw[black, pat1] (5.5,2.7)--(4.5,1.7)--(6.5,1.7)--(5.5,2.7);      
	        \node[rootvertex]  at (5.5,2.7) {};
		        \node at (5.5,3) {$r_{T}=r_{T_1}$};
	        \node at (5,1.4) {$\ell_{T_1}$};
	        \node at (5.5,2.1) {$T_1$};	        
	        \draw  (3.75,1)--(5,1.7)--(6.25,1);
	        \node[rleaf]   at (5,1.7) {}; 
	        \node at (3.3,1) {$r_{T_2}$};	        	        
	        \draw[black, pat2] (3.75,1)--(2.75,0)--(4.75,0)--(3.75,1);      
	        \node[rootvertex]  at (3.75,1) {};
	        \node at (3.75,0.35) {$T_2$};  
	        \node at (5.85,1) {$r_{T_3}$};	        	        	        
	        \draw[pat3] (6.25,1)--(5.25,0)--(7.25,0)--(6.25,1);      
	        \node[rootvertex]  at (6.25,1) {};
	        \node at (6.25,.35) {$T_3$};    
	        \node at (5,-0.5) {$T=T_1[T_2\oplus T_3]$};      
	        \draw[dashed] (8,-1)--(8,3.9);	
		\draw[draw=gray!40,rounded corners=4pt,  pat1] (10.75,3.75)--(10,2.9)--(11.5,2.9)--cycle;    
		\draw[draw=gray!40,rounded corners=4pt,  pat2] (9.75,2.75)--(9.3,2.15)--(10.2,2.15)--cycle;    
		\draw[draw=gray!40,rounded corners=4pt,  pat3] (10.75,2.75)--(10.3,2.15)--(11.2,2.15)--cycle;    				
		\node at (11.4,2.5) {$F_3$};							
		\node at (9.1,2.5) {$F_2$};	
		\node at (9.75,3.4) {$F_1$};		
	        \node at (10.75,3.8) {$r_{F}$};	  
	        \draw (10.25,3.05)--(10.5,3.3);
	        \draw (10.5,3.3)--(10.75,3.55)--(11.25,3.05);
	        \draw (10.75,3.05)--(10.5,3.3); 
	        \draw (9.5,2.3)--(9.75,2.55)--(10,2.3);
	       \draw (9.75,2.55)--(10.25,3.05)--(11,2.3);
	       \draw (10.72,2.55)--(10.5,2.3);	              
	        \node[invertex]  at (10.25,3.05) {}; 
	        \node[wleaf]   at (10.75,3.05) {}; 
	        \node[wleaf]  at (11.25,3.05) {}; 
	        \node[invertex]  at (10.5,3.3) {}; 
	        \node at (10,3.05) {$v$};    
	        \node at (9.5,2.525) {$x$};  
	        \node at (10.5,2.55) {$y$};      	        	                	        	          	   	            	        	                	        	          	        	                	        
	        \node[rootvertex]  at (10.75,3.55) {};
	        \node[invertex]  at (9.75,2.55) {}; 
	        \node[invertex]  at (10.75,2.55) {}; 
		  \node[wleaf]  at (9.5,2.3) {}; 
		  \node[wleaf]  at (10,2.3) {}; 
		  \node[wleaf]  at (10.5,2.3) {}; 
		  \node[wleaf]  at (11,2.3) {}; 
	        \node at (10,1.9) {the tree $F$};
	        \node at (10.5,1.2) {$r_{F_1}=r_F$};	        
	        \node at (9.5,0.3) {$y=\ell_{F_1}$};
	        \draw (10,0.5)--(10.5,1)--(11,0.5);
	        \draw (10.5,0.5)--(10.75,0.725);	        
	        \node[wleaf]   at (10.5,0.5) {}; 
	        \node[wleaf]  at (11,0.5) {}; 
	        \node[invertex]  at (10.75,0.725) {}; 
	        \node[rleaf]  at (10,0.5) {}; 
	        \node[rootvertex]  at (10.5,1) {};
	        \node at (9.5, .8) {$F_1$:};
		\node at (9.5,-.4) {$r_{F_2}=x$};	  
	        \draw (9.75,-0.95)--(9.5,-0.7)--(9.25,-0.95);		      
	        \node[wleaf]  at (9.75,-0.95) {}; 
	        \node[wleaf]  at (9.25,-0.95) {}; 
	       \node[rootvertex]  at (9.5,-0.7) {};	        
	        \node at (8.8, -0.8) {$F_2$:};
		\node at (11.2,-0.4) {$r_{F_3}=y$};
	        \draw (10.95,-0.95)--(11.2,-0.7)--(11.45,-0.95);			        
	        \node[wleaf]  at (11.45,-0.95) {}; 
	        \node[wleaf]  at (10.95,-0.95) {}; 
	       \node[rootvertex]  at (11.2,-0.7) {};	        
	        \node at (10.5, -0.8) {$F_3$:};

\end{tikzpicture}
\caption{Illustration of Lemma~\ref{lm:unhalf}.}\label{fig:unhalf}
\end{figure}
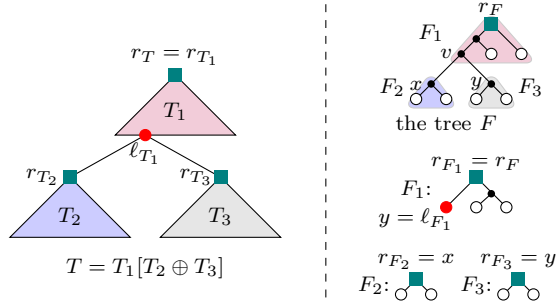

\begin{proof} 
Let  $T_1\in\mathcal{U}_1\left(\left\lfloor\frac{n}{2}\right\rfloor\right)$ and $T_2,T_3\in\mathcal{U}\left(\left\lfloor\frac{n}{2}\right\rfloor\right)$
Let $F$ be a rooted binary tree with $n$ leaves (and consequently $2n-2$ edges). Let $v$ be a white leaf centroid vertex of $F$, and let $x,y$ be the children of $v$. Let $F_1$ be the r-$2$-r tree obtained from $F$ by removing the descendants of $v$ and setting $\ell_{F_1}=v$, and let $F_2$ and $F_3$ be the subtrees rooted in $x$ and $y$, respectively. By Lemma~\ref{lm:balancer}, for each $i\in\{1,2,3\}$, we have 
$\whiteleaf{F_i}\le\lfloor\frac{n}{2}\rfloor$, and consequently
$F_i$ is an induced binary subtree of $T_i$. Note that $F_2,F_3$ each have at least one leaf and if $v=r_F$, then $F_1$ consists of a single red leaf and does not have white leaves.
As $F=F_1[F_2\oplus F_3]$, $F$ is an induced binary subtree of $T_1[T_2\oplus T_3]$. In particular, $T_1[T_2\oplus T_3] \in \mathcal{U}(n)$, which completes the proof.
\end{proof}

\begin{lemma}\label{lm:u1nhalf} 
Let $n\ge 2$, $T_1,T_2\in\mathcal{U}_1\left(\left\lfloor\frac{n}{2}\right\rfloor\right)$ and $T_3\in\mathcal{U}(n)$. Then, $T_1[T_2\oplus T_3]\in\mathcal{U}_1(n)$, and consequently
\begin{eqnarray}
u_1(n)&\leq& 2u_1(\lfloor n/2 \rfloor) +u(n) \label{eq:u1nhalfstart}\\
&\leq&3u_1(\lfloor n/2 \rfloor) +2u(\lfloor n/2 \rfloor). \label{eq:u1nhalf}
\end{eqnarray}
\end{lemma}

\begin{figure}[htb]
\centering
\begin{tikzpicture}[scale=.8,font=\tiny]
	        \draw[black, pat1] (5.5,1.75)--(4.5,0.75)--(6.5,0.75)--(5.5,1.75);      
	        \node[rootvertex]  at (5.5,1.75) {};
		\node at (5.5,1.9) {$r_{T}=r_{T_1}$};
	        \node at (5,0.45) {$\ell_{T_1}$};
	        \node at (5.5,1) {$T_1$};	        
	        \draw  (3.75,.25)--(5,0.75)--(6.25,.25);
	        \node[rleaf]   at (5,0.75) {}; 
	        \node at (3.3,.25) {$r_{T_2}$};	        	        
	        \draw[black, pat2] (3.75,.25)--(2.75,-0.75)--(4.75,-0.75)--(3.75,.25);      
	        \node[rootvertex]  at (3.75,.25) {};
		\node[rleaf]   at (3.25,-0.75) {}; 
	        \node at (3.25,-1) {$\ell_{T_2}$};	        	     		
	        \node at (3.75,-0.5) {$T_2$};  
	        \node at (5.85,.25) {$r_{T_3}$};	        	        	        
	        \draw[black, pat3] (6.25,.25)--(5.,-1)--(7.5,-1)--(6.25,.25);      
	        \node[rootvertex]  at (6.25,.25) {};
	        \node at (6.25,-0.5) {$T_3$};      
	        \node at (5,-1.5) {$T=T_1[T_2\oplus T_3]$};      
	        \draw[dashed] (8,-3.5)--(8,3.5);	
		\draw[draw=gray!40,rounded corners=4pt,  pat1] (11.5,3.2)--(10.3,1.9)--(12.7,1.9)--cycle;    
		\draw[draw=gray!40,rounded corners=4pt,  pat2] (10.25,1.85)--(9.2,.9)--(11.3,.9)--cycle;    
		\node[circle, draw=gray!40, pat3, inner sep=3.5pt] at (11.5,1) {};			
		\node at (12,1) {$F_3$};							
		\node at (9.3,1.5) {$F_2$};	
		\node at (10.5,2.7) {$F_1$};		
	        \node at (11.5,3.25) {$r_{F}$};
	        \draw (11,2)--(11.25,2.25)--(11.5,2);     	        	                	        	          	   	           	        	                	        	          	        	                	        
	       \draw (11.75,2.75)--(11.5,3)--(11,2.5);
	        \draw (11.5,2.5)--(11.74,2.75)--(12,2.5);
	        \draw (10.5,2)--(11,2.5)--(11.25,2.25); 
	        \draw (10.25,1.75)--(10.5,2)--(11.5,1);	   
	       \draw (9.75,1.25)--(10.25,1.75)--(11,1);
	       \draw (10.72,1.25)--(10.5,1);	       
	        \draw (9.5,1)--(9.75,1.25)--(10,1);	             	  
	        \node[invertex]   at (11.75,2.75) {}; 
	        \node[wleaf]   at (12,2.5) {}; 
	        \node[wleaf]   at (11.5,2.5) {}; 
	        \node[invertex]  at (11,2.5) {}; 
	        \node[invertex]  at (11.25,2.25) {}; 
	        \node[wleaf]  at (11,2) {}; 
	        \node[wleaf]  at (11.5,2) {}; 
	        \node at (10.25,2) {$v$};    
	        \node at (10,1.75) {$x$};  
	        \node at (11.5,.7) {$y$}; 
	        \node[invertex]  at (10.25,1.75) {}; 
	        \node[invertex]  at (10.5,2) {}; 
	        \node[rootvertex]  at (11.5,3) {};
	        \node[wleaf]  at (9.5,1) {}; 
	        \node[invertex]  at (9.75,1.25) {}; 
	        \node[invertex]  at (10.75,1.25) {}; 
		  \node[wleaf]  at (9.5,1) {}; 
		  \node[wleaf]  at (10,1) {}; 
		  \node[wleaf]  at (10.5,1) {}; 
	        \node[wleaf]  at (11.5,1) {}; 
		  \node[rleaf]  at (11,1) {}; 
	        \node at (11,.7) {$\ell_F$}; 
		         
	        \node at (10.5,.3) {$\ell_F$ is a descendant of $v$};
	        \node at (11,-.25) {$r_{F_1}$};	  
	        \draw (10.5,-1.5)--(10.75,-1.25)--(11,-1.5);     	        	                	        	          	   	   	        	                	        	          	        	                
	        \draw (11.25,-.75)--(11,-.5)--(10.5,-1);
	        \draw (11,-1)--(11.25,-.75)--(11.5,-1);
	        \draw (10,-1.5)--(10.5,-1)--(10.75,-1.25);         	        
	        \node[invertex]   at (11.25,-.75) {}; 
	        \node[wleaf]   at (11.5,-1) {}; 
	        \node[wleaf]   at (11,-1) {}; 
	        \node[invertex]  at (10.5,-1) {}; 
	        \node[invertex]  at (10.75,-1.25) {}; 
	        \node[wleaf]  at (10.5,-1.5) {}; 
	        \node[wleaf]  at (11,-1.5) {}; 
	       \node[rleaf]  at (10,-1.5) {}; 
	       \node[rootvertex]  at (11,-.5) {};
	        \node at (9.75,-1.8) {$v=\ell_{F_1}$}; 
	        \node at (10, -0.85) {$F_1$};
	        \node at (10.25,-2.35) {$r_{F_2}=x$};	
	        \draw (9.5,-3.25)--(9.75,-3)--(10,-3.25);	         
	       \draw (10.5,-3.25)--(10.75,-3)--(11,-3.25);
	       \draw (9.75,-3)--(10.25,-2.5)--(10.75,-3);	         		       
	        \node[invertex]  at (9.75,-3) {}; 
	        \node[invertex]  at (10.75,-3) {}; 
		  \node[wleaf]  at (9.5,-3.25) {}; 
		  \node[wleaf]  at (10,-3.25) {}; 
		  \node[wleaf]  at (10.5,-3.25) {}; 
	        \node[rootvertex]  at (10.25,-2.5) {};	       
		  \node[rleaf]  at (11,-3.25) {}; 
	        \node at (11.5,-3.5) {$\ell_{F_2}=\ell_F$}; 
	        \node at (9.4,-2.65) {$F_2$}; 
	        \node at (12.2,-1.95) {$r_{F_3}=y$};	  				
	        \node[whiteroot]  at (12.2,-2.25) {};	       
	        \node at (11.7,-2.25) {$F_3$}; 
	        \draw[dashed] (13.5,-3.5)--(13.5,3.5);	
		\draw[draw=gray!40,rounded corners=4pt,  pat1] (17,3.2)--(16.3,2.4)--(17.7,2.4)--cycle;    
		\draw[draw=gray!40,rounded corners=4pt,  pat2] (16.75,2.4)--(16.3,1.85)--(17.15,1.85)--cycle;   
		\draw[draw=gray!40,rounded corners=4pt,  pat3] (16,2.2)--(14.8,.9)--(17.2,.9)--cycle;    
		\node at (14.8,1.5) {$F_3$};							
		\node at (17.2,2.15) {$F_2$};	
		\node at (16.26,2.9) {$F_1$};		
	        \node at (17,3.25) {$r_{F}$};	
	        \draw (16.5,2)--(16.75,2.25)--(17,2);     	      	        
	        \draw (17.25,2.75)--(17,3)--(16.5,2.5);
	        \draw (17,2.5)--(17.25,2.75)--(17.5,2.5);
	        \draw (16,2)--(16.5,2.5)--(16.75,2.25); 
	        \draw (15.75,1.75)--(16,2)--(17,1);	
	        \draw (15,1)--(15.25,1.25)--(15.5,1);
	       \draw (15.25,1.25)--(15.75,1.75)--(16.5,1);
	       \draw (16.22,1.25)--(16,1);	        	                  
	        \node[invertex]   at (17.25,2.75) {}; 
	        \node[wleaf]   at (17.5,2.5) {}; 
	        \node[wleaf]   at (17,2.5) {}; 
	        \node[invertex]  at (16.5,2.5) {}; 
	        \node[invertex]  at (16.75,2.25) {}; 
	        \node[wleaf]  at (16.5,2) {}; 
	        \node at (15.25,2) {$v=y$};    
	        \node at (16.5,2.25) {$x$};  
	        \node at (16.1,2.5) {$z$}; 
	        \node[rleaf]  at (17,2) {}; 
	        \node[invertex]  at (15.75,1.75) {}; 
	        \node[invertex]  at (16,2) {}; 
	        \node[rootvertex]  at (17,3) {};
	        \node[wleaf]  at (15,1) {}; 
	        \node[invertex]  at (15.25,1.25) {}; 
	        \node[invertex]  at (16.25,1.25) {}; 
		  \node[wleaf]  at (15,1) {}; 
		  \node[wleaf]  at (15.5,1) {}; 
		  \node[wleaf]  at (16,1) {}; 
	        \node[wleaf]  at (17,1) {}; 
		  \node[wleaf]  at (16.5,1) {}; 
	        \node at (16.5,.3) {$\ell_F$ is not a descendant of $v$};
	        \node at (16,-.25) {$r_{F_1}$};	  
	        \draw (16.5,-1)--(16,-.5)--(15.5,-1);
	        \draw (16,-1)--(16.25,-.75)--(16.5,-1);        	        
	        \node[invertex]   at (16.25,-.75) {}; 
	        \node[wleaf]   at (16.5,-1) {}; 
	        \node[wleaf]   at (16,-1) {}; 
	       \node[rleaf]  at (15.5,-1) {}; 
	       \node[rootvertex]  at (16,-.5) {};
	        \node at (15,-1.3) {$z=\ell_{F_1}$}; 
	        \node at (15.3, -.5) {$F_1$};
	        \node at (17.85,-1.5) {$r_{F_2}=x$};	 
	       \draw (17.6,-2)--(17.85,-1.75)--(18.1,-2);  	         
	        \node[wleaf]   at (17.6,-2) {}; 
	       \node[rootvertex]  at (17.85,-1.75) {};	
	        \node[rleaf]   at (18.1,-2) {}; 
	        \node at (17.2,-1.8) {$F_2$};	               	        	          	   	            	        	      		
	        \node at (16,-2.3) {$r_{F_3}=y$};	  		
               \draw (15.75,-2.75)--(16,-2.5)--(17,-3.5);
	        \draw (15,-3.5)--(15.25,-3.25)--(15.5,-3.5);               
	       \draw (15.25,-3.25)--(15.75,-2.75)--(16.5,-3.5);
	       \draw (16.25,-3.25)--(16,-3.5);
	        \node[invertex]  at (15.75,-2.75) {}; 
	        \node[invertex]  at (15.25,-3.25) {}; 
	        \node[invertex]  at (16.25,-3.25) {}; 
		  \node[wleaf]  at (15,-3.5) {}; 
		  \node[wleaf]  at (15.5,-3.5) {}; 
		  \node[wleaf]  at (16,-3.5) {}; 
	        \node[wleaf]  at (17,-3.5) {}; 
	        \node[rootvertex]  at (16,-2.5) {};	       
		  \node[wleaf]  at (16.5,-3.5) {}; 
	 	        \node at (15,-2.75) {$F_3$}; 
\end{tikzpicture}
\caption{Illustration of Lemma~\ref{lm:u1nhalf}. Note that while in this example, we have $v=y$ in the case that $\ell_F$ is not a descendant of $v$, this is not true in general.}\label{fig:u1nhalf}
\end{figure}
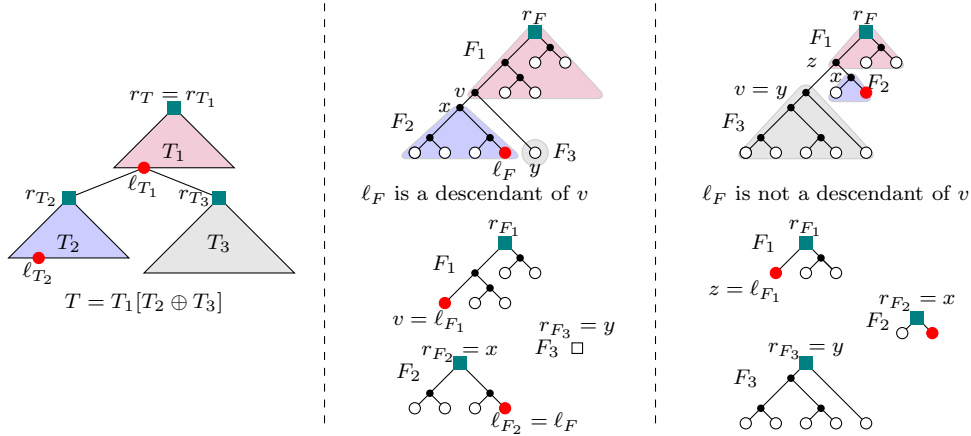

\begin{proof}
It is enough to prove Equation~\eqref{eq:u1nhalfstart}, as Equation~\eqref{eq:u1nhalf} then follows from Lemma~\ref{lm:unhalf}.
Let $F$ be a redleaf tree with $n$ white leaves, and let $v$ be a white leaf centroid vertex of $F$.

If $\ell_F$ is a descendant of $v$, let $x,y$ be the two children of $v$ such that $x$  is 
on the $v$-$\ell_F$ path. Let $F_1$ be the redleaf tree
obtained from $F$ by removing the descendants of $v$ and setting $\ell_{F_1}=v$, and let $F_2,F_3$ be the subtrees of $F$ rooted in $x$ and $y$, respectively. We have that $F_2$ is a redleaf tree. By Lemma~\ref{lm:balancer}, for each $i\in\{1,2,3\}$, we have  $\whiteleaf{F_i}\le\lfloor\frac{n}{2}\rfloor$,
so $F_i$ is an induced binary subtree of $T_i$. As $F=F_1[F_2\oplus F_3]$, $F$ is an induced binary subtree of $T_1[T_2\oplus T_3]$.

So assume $\ell_F$ is not a descendant of $v$. Let $z$ be the last common ancestor of $\ell_F$ and $v$,
and let $x,y$ be the two children of $z$ such that $x$ is on the $z$-$\ell_F$ path. Let
$F_1$ be the redleaf tree obtained from $F$ by removing the descendants of $z$ and setting $\ell_{F_1}=z$, let
$F_2$ be the redleaf tree rooted in $x$, and let $F_3$ be the subtree rooted in $y$. As $\whiteleaf{F_1}+\whiteleaf{F_2}\le\lfloor\frac{n}{2}\rfloor$ and $\whiteleaf{F_3} \le \lfloor\frac{n}{2}\rfloor$ by Lemma \ref{lm:balancer}, we
have that for each $i\in\{1,2,3\}$ $F_i$ is an induced binary subtree of $T_1$. 
As $F=F_1[F_2\oplus F_3]$, $F$ is an induced binary subtree of $T_1[T_2\oplus T_3]$. In particular, $T_1[T_2\oplus T_3] \in \mathcal{U}_1(n)$, which completes the proof.
\end{proof}

The preceding lemmas combine  into the following recursions in two unknown functions:
\begin{eqnarray*}
   u(n) & \leq & u_1(\lfloor n/2 \rfloor)+2u(\lfloor n/2 \rfloor) \\
   u_1(n) & \leq & 3u_1(\lfloor n/2 \rfloor) + 2u(\lfloor n/2 \rfloor).
\end{eqnarray*}
We will apply the recursions above repeatedly to get upper bounds for $u(n)$, which will lead to
two sequences $\{a_i\}$ and $\{b_i\}$ forming the coefficients of the functions $u_1$ and $u$ in the resulting upper bound.

\begin{lemma}\label{lm:halfsequence} Let the sequences $a_1,a_2,\ldots$ and $b_1,b_2,\ldots$ be given by $a_1=1$, $b_1=2$, and for $n\ge 2$ $a_{n}=3a_{n-1}+b_{n-1}$
and $b_{n}=2a_{n-1}+2b_{n-1}$. Then $b_n=\frac{5}{6}\cdot 4^n-\frac{4}{3}$.
\end{lemma}

\begin{proof} 
Notice that for $n\ge 2$, we have $2a_n=3b_n-4b_{n-1}$. Therefore, for $n\ge 3$, we have $b_n=5b_{n-1}-4b_{n-2}$, where $b_1=2$ and $b_2=6$. Standard techniques give that for some constants $B,C$, we have $b_n = B \cdot 4^n + C \cdot 1^n$. Using initial terms in the sequence, we find $b_n= \frac{1}{3}\cdot 4^n+\frac{2}{3}$.
\end{proof}

\begin{theorem}\label{thm:upperbound} 
We have $u(n)\le\frac{2}{3}n^2+\frac{1}{3}=O(n^2).$
\end{theorem}

\begin{proof}
Observe that $u(n)$ and $u_1(n)$ are monotone increasing, and for $0<\alpha<1$ we have $\lfloor\alpha\lfloor\alpha n\rfloor\rfloor\le  \lfloor\alpha^2 n\rfloor$. 
Combining these facts with the recursive upper bounds above
we will show by induction on $k$ that, as long as $n\ge 2^k$, we have
\begin{equation}
u(n)\le a_k u_1\left(\left\lfloor2^{-k} n\right\rfloor\right)+b_ku\left(\left\lfloor2^{-k} n\right\rfloor\right), \label{eq:toprove}
\end{equation}
where sequences $a_i$ and $b_i$ are as defined in Lemma~\ref{lm:halfsequence}.

For $n\ge 2$, $u(n)  \leq  u_1(\lfloor n/2 \rfloor)+2u(\lfloor n/2 \rfloor) =a_1u_1(\lfloor 2^{-1}n \rfloor)+b_1u(\lfloor 2^{-1}n \rfloor) $, so the statement is true for $k=1$.

Let $k\ge 1$ and assume Equation~\eqref{eq:toprove} holds for $k$. If $n\ge 2^{k+1}$, i.e.,
$2^{-k}n\ge 2$, then the induction hypothesis and the recursive formulas above yield 
\begin{eqnarray*}
u(n)&\le& a_k u_1\left(\left\lfloor2^{-k} n\right\rfloor\right)+b_ku\left(\left\lfloor2^{-k} n\right\rfloor\right)\\
&\le& 
a_k\left(3u_1\left(\left\lfloor\lfloor 2^{-k}n \rfloor/2\right\rfloor\right) 
+ 2u\left(\left\lfloor \lfloor 2^{-k}n \rfloor/2 \right\rfloor\right)\right)
+b_k\left(u_1\left(\left\lfloor\left\lfloor2^{-k} n\right\rfloor/2\right\rfloor\right)
+2u\left(\left\lfloor \left\lfloor2^{-k} n\right\rfloor/2\right\rfloor\right)
\right)\\
&\le& (3a_k+b_k)u_1\left(\lfloor 2^{-(k+1)}n \rfloor\right)+
(2a_k+2b_k)u\left(\lfloor 2^{-(k+1)}n \rfloor\right)\\
&=&a_{k+1}u_1\left(\lfloor 2^{-(k+1)}n \rfloor\right)+
b_{k+1}u\left(\lfloor 2^{-(k+1)}n \rfloor\right),
\end{eqnarray*}
which finishes the proof of Inequality~\eqref{eq:toprove}.

Using the facts that $u_1(1)=u(1)=1$, this gives that if $m$ is the smallest integer such that $2^{-m} n<2 $, then 
$u(n)\le a_{m-1}+b_{m-1}=\frac{1}{2}b_{m}=\frac{1}{6} \cdot 4^{m}+\frac{1}{3}$.  
Since $m$ is minimal such that $n<2^{m+1}$, $m=\lfloor \log_2(n)\rfloor+1\le\log_2(n)+1$. Therefore, $4^m\le 4n^2$, which gives $u(n)\le \frac{2}{3}n^2+\frac{1}{3}$.
\end{proof}

\subsection{A superlinear lower bound}

In this section, we will establish a superlinear lower bound on the growth rate of $u(n)$.
\begin{theorem}\label{thm:lowerbound} The following holds:
\[u(n)=\Omega(n\log n).\]
\end{theorem}
\begin{proof}
As $u(n)$ is increasing, it is enough to prove that $u(2^k)=\Omega(k2^k)$. 
First, we recall Kai Lai  Chung's inclusion-exclusion formula  \cite{kailai,galambos}  for the union of finite sets. For any positive integer $\ell$ and finite sets $A_1,\ldots,A_m$, we have
\begin{equation} \label{eq:oskai} 
\binom{\ell+1}{2} \left\vert \bigcup_{i=1}^m A_i\right\vert\geq  \ell \sum_{i=1}^m \left\vert A_i\right\vert -\sum_{1\leq i<j\leq m} \left\vert A_i\cap A_j\right\vert.  
\end{equation}

Let $k\ge 2$ and $n=2^k$.
For $i:1\le i\le k$, let $T_i=J_{i-1, 2^{k-i+1}}$. Then $\whiteleaf{T_i}= 2^{i-1} \cdot 2^{k-i+1} = 2^k = n$.
Fix an $n$-universal tree $U^{\star}$. Let the universe be the set of leaves of $U^{\star}$, and for each $i:1\le i\le k$, let $A_i$ be an $n$-element   
subset of the universe that induces $T_i$ in the universal tree $U^{\star}$. 
By Lemma~\ref{lm:jelly}, for $i<j$,  mast$(T_i,T_j)=2^{i-1}(2^{k-j+1}+j-i)=2^{k-(j-i)}+2^{i-1}(j-i)$.
Kai Lai Chung's formula \eqref{eq:oskai} with $\ell=2$ and $m=k$  gives that 
\begin{equation} \label{eq:ourkai}
u\left(2^k\right) \geq \frac{2k2^k}{3} -\frac{1}{3}\sum_{1\leq i<j\leq k} \mast (T_i,T_j).
\end{equation} 

We also see that 
\begin{eqnarray*} 
\sum_{1\leq i<j\leq k} \mast (T_i,T_j)&=& 
\sum_{i=1}^{k-1}\sum_{j=i+1}^{k} \left(2^{k-(j-i)}+2^{i-1}(j-i)\right)
=\sum_{i=1}^{k-1}\sum_{j=1}^{k-i}\left(2^{k-j}+2^{i-1}j\right)\\
&=&\sum_{j=1}^{k-1}\sum_{i=1}^{k-j}\left(2^{k-j}+2^{i-1}j\right)\leq \sum_{j=1}^{k-1}\left((k-j)2^{k-j}+j2^{k-j}\right)\\
&=&\sum_{j=1}^{k-1}k2^{k-j}=k\sum_{j=1}^{k-1}2^{k-j}<k 2^k
\end{eqnarray*} 

Using this in Inequality~\eqref{eq:ourkai}, we obtain
\begin{equation}
u\left(2^k\right) \geq \frac{2k2^k}{3} - \frac{k2^k}{3} =\frac{k2^k}{3}=\Omega(k2^k). \qedhere 
\end{equation} 
\end{proof}

\subsection{r-\texorpdfstring{$2$}{2}-r trees}
In this subsection, we will connect the (unknown) growth rates of $u(n)$ and $u_1(n)$. 

\begin{definition}\label{def:svec} For each $k\in\mathbb{N}$ we define a sequence $\vec{s}_k=(s_{1,k},s_{2,k},\ldots,s_{2^{k+1}-1,k})$ as follows:  $\vec{s}_0=(1)$, and for
each $k>1$, $\vec{s}_k$ is the concatenation of $\vec{s}_{k-1}, 2^k$, and
$\vec{s}_{k-1}$, i.e., $s_{2^k,k}=2^k$ and $(s_{1,k},\ldots,s_{2^{k}-1,k})=(s_{2^k+1,k},\ldots,s_{2^{k+1}-1,k})=\vec{s}_{k-1}$. It is easy to see that for each $k\in\mathbb{N}$ the sequence $\vec{s}_k$ contains the term $2^i$ exactly $2^{k-i}$ times for each $i:0\le i\le k$.
\end{definition}

\begin{lemma}\label{lm:ssub} Let $n=2^k$ and assume $b_1+b_2+\cdots +b_{\ell}$ is a composition of $n$ into positive terms. Then, there are integers $1\le a_1<a_2<\cdots a_{\ell}\le 2n-1$
such that $b_i \leq s_{a_i,k}$ for each $i:1\le i\le\ell$. 
\end{lemma}

\begin{proof} We prove this by induction on $k$. The base case $k=0$ is obvious.
Let $k\ge 1$ and assume that the statement is true for $k-1$. Let $d$ be the smallest integer such that $\sum_{i=1}^d b_i>\frac{n}{2}=2^{k-1}$, and set $a_d=2^k=n$.
As $\sum_{i=1}^{\ell} b_i=n=2^k>2^{k-1}$, $d$ is well defined.
Then, we have $\sum_{i=1}^{d-1} b_i\le 2^{k-1}$ and $\sum_{i=d+1}^{\ell}b_i<2^{k-1}$. The sequences
$(b_1,\ldots,b_{d-1})$ and $(b_{d+1},\ldots,b_{\ell})$ embed into $\vec{s}_{k-1}$ by the induction hypothesis (as we can extend them by one more positive term if needed
if their sum is smaller than $2^{k-1}$) and $b_d \leq 2^k$, and together this gives the required sequence $a_i$.
\end{proof}

\begin{lemma}\label{lm:compare} For every $k\in\mathbb{N}$, $u_1(2^k)\le\sum_{i=0}^k 2^i u(2^{k-i})$.
\end{lemma}

\begin{proof} 
Fix a $k\in\mathbb{N}$, and for each $i:1\le i\le 2^{k+1}-1$, let  $R_i$ be an $s_{i,k}$-universal tree of size $u(s_{i,k})$, and let $Q$ be the redleaf tree with no white leaves. Set
$T=R_{1}\oplus(R_{2}\oplus(R_{3}\oplus(\cdots(R_{2^{k+1}-1}\oplus Q)\cdots)))$. Then $T$ is a redleaf tree with
$\whiteleaf{T}=\sum_{i=0}^k 2^i u(2^{k-i})$. We will show that $T\in\mathcal{U}_1(2^k)$, which finishes the proof. 

Let $F$ be a redleaf tree and let $P=x_1x_2\ldots x_{\ell}x_{\ell+1}$ be the unique $r_F$-$\ell_F$ path in $F$ (so $x_1=r_F$ and $x_{\ell+1}=\ell_F$). Clearly, $1\le\ell\le n$. For each
$i:1\le i\le \ell$ let $y_i$ be the child of $x_i$ that is not $x_{i+1}$ and let $F_i$ be the subtree of $F$ rooted at $y_i$. Then 
$F=F_{1}\oplus(F_{2}\oplus(F_{3}\oplus(\cdots(F_{\ell}\oplus Q)\cdots)))$, and the statement follows from setting 
$b_i=\whiteleaf{F_i}$ and Lemma~\ref{lm:ssub}. 
\end{proof}

Using the above results, we now relate bounds between $u(n)$ and $u_1(n)$. In particular, we will see that a polynomial upper bound directly translates between the two. While a non-polynomial upper bound does not translate from $u(n)$ to $u_1(n)$, one can still obtain some information on $u_1(n)$ based on $u(n)$.

\begin{theorem}\label{thm:growth}
For any $\alpha>1$, we have $u(n)=O(n^{\alpha})$ if and only if $u_1(n)=O(n^{\alpha})$.
\end{theorem}

\begin{proof} As both $u$ and $u_1$ are increasing, it is enough to prove the statements for leaf numbers that are powers of two, i.e., for the subsequence $n_i=2^i$. 
The latter is immediate from Lemma~\ref{lm:compare} and the obvious fact that $u(n)\le u_1(n)$.
\end{proof}

\begin{corollary}
   The following hold:
\begin{enumerate}[label={\upshape (\roman*)}]
       \item $u_1(n)=O(n^2)$, and
       \item $u_1(n)=\Omega(n\log n)$. 
   \end{enumerate}
\end{corollary}

\begin{proof}
   The upper bound follows from Theorem \ref{thm:upperbound} and Theorem \ref{thm:growth}. The lower bound follows from  Theorem \ref{thm:lowerbound} and the fact $u(n)\leq u_1(n)$. 
\end{proof}

\section{Computational results} \label{sec:pics}

For $n\in\{1,2,3\}$, the rooted binary tree with $n$ leaves is unique, and $u(n)=n$. In the following pictures, $U_n$ will denote the $n$-universal tree of minimal size (if it is unique) and
$U_n^i$ will denote the $n$-universal trees of minimal size, where $i$ ranges between 1 and the total number of $n$-universal trees of minimal size.

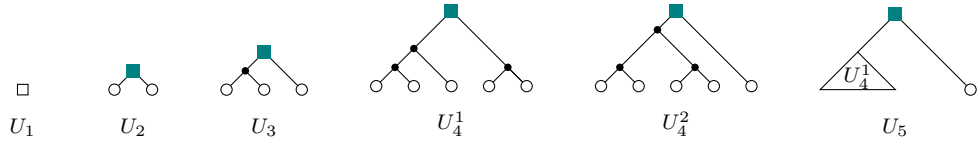
\begin{figure}[H]
\centering
\begin{tikzpicture}[scale=.5,font=\tiny]    
   \node[whiteroot]  at (0,0) {};      
\node at (0,-1) {$U_1$};  
\end{tikzpicture}\quad\quad
\begin{tikzpicture}[scale=.5,font=\tiny]    
\draw (0,0) -- (0.5,0.5) -- (1,0);
   \node[wleaf]  at (0,0) {};      
   \node[wleaf]  at (1,0) {};      
   \node[rootvertex]  at (0.5,0.5) {};
\node at (.5,-1) {$U_2$};  
\end{tikzpicture}\quad\quad
\begin{tikzpicture}[scale=.5,font=\tiny]    
\draw (0,0) -- (0.5,0.5) -- (1,0);
   \draw (0.5,0.5) -- (1,1) -- (2,0);
   \node[wleaf]  at (0,0) {};      
   \node[wleaf]  at (1,0) {};      
   \node[wleaf]  at (2,0) {}; 
   \node[invertex]  at (0.5,0.5) {};
   \node[rootvertex]  at (1,1) {};
\node at (1,-1) {$U_3$};  
\end{tikzpicture}\quad\quad
\begin{tikzpicture}[scale=.5,font=\tiny]    
\draw (0,0) -- (0.5,0.5) -- (1,0);
   \draw (0.5,0.5) -- (1,1) -- (2,0);
   \draw (3,0) -- (3.5,0.5) -- (4,0);
   \draw (1,1) -- (2,2) -- (4,0);
   \node[wleaf]  at (0,0) {};      
   \node[wleaf]  at (1,0) {};      
   \node[wleaf]  at (2,0) {}; 
   \node[wleaf]  at (3,0) {}; 
   \node[wleaf]  at (4,0) {}; 
   \node[invertex]  at (0.5,0.5) {};
   \node[invertex]  at (1,1) {};
   \node[rootvertex]  at (2,2) {};
   \node[invertex]  at (3.5,0.5) {};
\node at (2,-1) {$U_4^1$};  
\end{tikzpicture}\quad\quad
\begin{tikzpicture}[scale=.5,font=\tiny]    
\draw (0,0) -- (0.5,0.5) -- (1,0);
   \draw (2.5,0.5) -- (2,0);
   \draw (1.5,1.5) -- (3,0);
   \draw (0.5,0.5) -- (2,2) -- (4,0);
   \node[wleaf]  at (0,0) {};      
   \node[wleaf]  at (1,0) {};      
   \node[wleaf]  at (2,0) {}; 
   \node[wleaf]  at (3,0) {}; 
   \node[wleaf]  at (4,0) {}; 
   \node[invertex]  at (0.5,0.5) {};
   \node[rootvertex]  at (2,2) {};
   \node[invertex]  at (1.5,1.5) {};    
   \node[fill=black,circle,inner sep=1pt]  at (2.5,0.5) {}; 
\node at (2,-1) {$U_4^2$};  
\end{tikzpicture}\quad\quad
\begin{tikzpicture}[scale=.5,font=\tiny]    
   \draw (1,1) -- (2,2) -- (4,0);
   \draw (1,1)--(0,0)--(2,0)--cycle;
   \node at (1,.35) {$U_4^1$};
   \node[wleaf]  at (4,0) {}; 
   \node[rootvertex]  at (2,2) {}; 
\node at (2,-1) {$U_5$};  
\end{tikzpicture}
\caption{The minimum size $n$-universal trees for $n\in\{1,2,3,4,5\}$.} 
\end{figure}

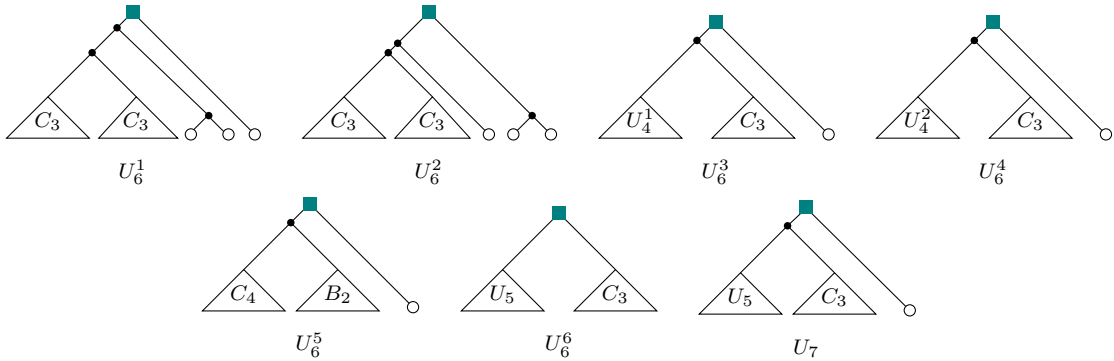
\begin{figure}[H]
\centering
\begin{tikzpicture}[scale=.5,font=\tiny] 
   \draw (-.1,-.1)--(2.1,-.1)--(1,1)--cycle;
   \node at (1,0.35) {$C_3$};  
   \draw (2.35,-.1)--(4.45,-.1)--(3.35,1)--cycle;
   \node at (3.35,0.35) {$C_3$};  
   \draw (1,1)--(3.25,3.25);
   \draw(2.175,2.175)--(3.35,1);
   \draw(2.8375,2.8375)--(5.275,0.5);
   \draw(4.75,0)--(5.25,0.5)--(5.75,0);
   \draw(3.25,3.25)--(6.5,0);
   \node[wleaf]  at (4.8,0) {};   
   \node[wleaf]  at (5.8,0) {}; 
   \node[wleaf]  at (6.5,0) {}; 
   \node[invertex]  at (2.175,2.175) {};
   \node[invertex]  at (2.8375,2.8375) {};
   \node[rootvertex]  at (3.25,3.25) {};
   \node[invertex]  at (5.275,0.5) {};
\node at (3.25,-1) {$U_6^1$};  
\end{tikzpicture}
\quad
\begin{tikzpicture}[scale=.5,font=\tiny]   
\draw (-.1,-.1)--(2.1,-.1)--(1,1)--cycle;
   \node at (1,0.35) {$C_3$};  
   \draw (2.35,-.1)--(4.35,-.1)--(3.35,1)--cycle;
   \node at (3.35,0.35) {$C_3$};  
   \draw (1,1)--(3.25,3.25);
   \draw(2.175,2.175)--(3.35,1);
   \draw(2.425,2.425)--(4.85,0);
   \draw(5.5,0)--(6,0.5)--(6.5,0);
   \draw(3.25,3.25)--(6,0.5);
   \node[wleaf]  at (4.85,0) {};   
   \node[wleaf]  at (5.5,0) {}; 
   \node[wleaf]  at (6.5,0) {}; 
   \node[invertex]  at (2.175,2.175) {};
   \node[invertex]  at (2.425,2.425) {};
   \node[rootvertex]  at (3.25,3.25) {};
   \node[invertex]  at (6,0.5) {};
   \node at (3.25,-1) {$U_6^2$};  
   \end{tikzpicture}
\quad
\begin{tikzpicture}[scale=.5, font=\tiny]  
\draw (1,1) -- (2.5,2.5) -- (4,1);
 \draw (2.5,2.5) -- (3,3) -- (6,0);
 \draw (1,1)--(-0.1,-.1)--(2.1,-.1)--cycle;
 \node at (1,0.35) {$U_4^1$};  
 \draw(2.9,-.1)--(4,1)--(5.1,-.1)--cycle;
 \node at (4,0.35) {$C_3$};  
 \node[wleaf]  at (6,0) {};
 \node[invertex]  at (2.5,2.5) {};
 \node[rootvertex]  at (3,3) {};
\node at (3,-1) {$U_6^3$};  \end{tikzpicture}
\quad
\begin{tikzpicture}[scale=.5,font=\tiny] 
 \draw (1,1) -- (2.5,2.5) -- (4,1);
 \draw (2.5,2.5) -- (3,3) -- (6,0);
 \draw (1,1)--(-0.1,-.1)--(2.1,-.1)--cycle;
 \node at (1,0.35) {$U_4^2$};  
 \draw(2.9,-.1)--(4,1)--(5.1,-.1)--cycle;
 \node at (4,0.35) {$C_3$};  
 \node[wleaf]  at (6,0) {};
 \node[invertex]  at (2.5,2.5) {};
 \node[rootvertex]  at (3,3) {};
\node at (3,-1) {$U_6^4$};  \end{tikzpicture}
\quad
\begin{tikzpicture}[scale=.5,font=\tiny] 
\draw (1,1) -- (2.25,2.25) -- (3.5,1);
 \draw (2.25,2.25) -- (2.75,2.75) -- (5.5,0);
 \draw (1,1)--(-0.1,-.1)--(2.1,-.1)--cycle;
 \node at (1,0.35) {$C_4$};  
 \draw(2.4,-.1)--(3.5,1)--(4.6,-.1)--cycle;
 \node at (3.5,0.35) {$B_2$};  
 \node[wleaf]  at (5.5,0) {};
 \node[invertex]  at (2.25,2.25) {};
 \node[rootvertex]  at (2.75,2.75) {};
\node at (2.75,-1) {$U_6^5$}; 
\end{tikzpicture}
\quad
\begin{tikzpicture}[scale=.5,font=\tiny] 
\draw (1,1) -- (2.5,2.5) -- (4,1);
 \draw (1,1)--(-0.1,-.1)--(2.1,-.1)--cycle;
 \node at (1,0.35) {$U_5$};  
 \draw(2.9,-.1)--(4,1)--(5.1,-.1)--cycle;
 \node at (4,0.35) {$C_3$};  
 \node[rootvertex]  at (2.5,2.5) {};
\node at (2.5,-1) {$U_6^6$};  
\end{tikzpicture}
\quad
\begin{tikzpicture}[scale=.5,font=\tiny] 
 \draw (1,1) -- (2.25,2.25) -- (3.5,1);
 \draw (2.25,2.25) -- (2.75,2.72) -- (5.5,0);
 \draw (1,1)--(-0.1,-.1)--(2.1,-.1)--cycle;
 \node at (1,0.35) {$U_5$};  
 \draw(2.4,-.1)--(3.5,1)--(4.6,-.1)--cycle;
 \node at (3.5,0.35) {$C_3$};  
 \node[wleaf]  at (5.5,0) {};
 \node[invertex]  at (2.25,2.25) {};
 \node[rootvertex]  at (2.75,2.75) {};
\node at (2.75,-1) {$U_7$};  
\end{tikzpicture}
\caption{The minimum size $n$-universal trees for $n\in\{6,7\}$.} 
\end{figure}

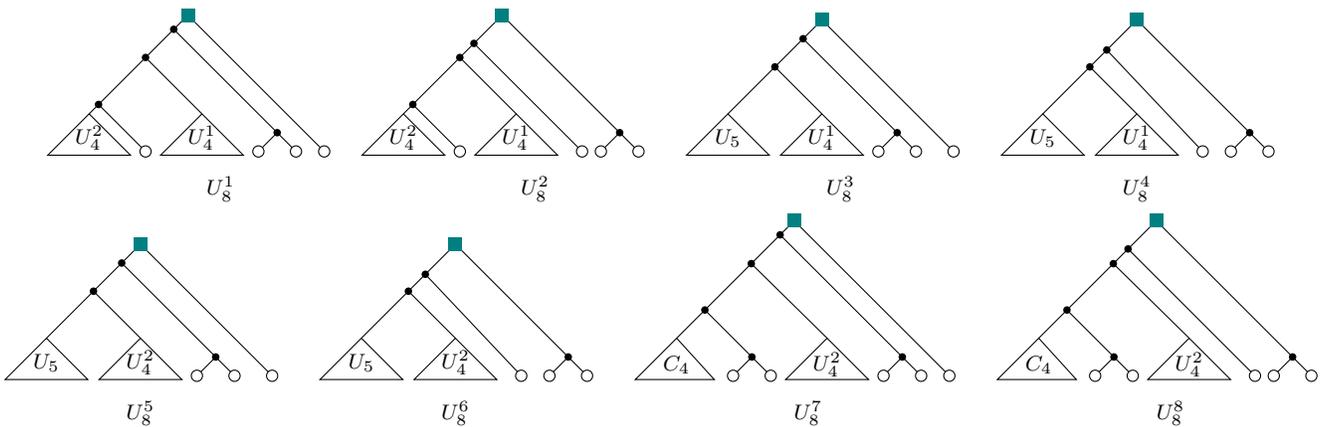
\begin{figure}[H]
\centering
\begin{tikzpicture}[scale=.5,font=\tiny,baseline={(0,0)}] 
   \draw (1.25,1.25) -- (2.5,0);
   \draw (1,1) -- (2,2) ;
   \draw (2,2) -- (2.5,2.5) -- (4,1);%
   \draw (2.5,2.5) -- (3.25,3.25) -- (6.5,0);%
   \draw (3.25,3.25) -- (3.625,3.625) -- (7.25,0);%
   \draw (5.5,0) -- (6,0.5);%
   \draw (1,1)--(-.1,-.1)--(2.1,-.1)--cycle;
   \node  at (1,0.35) {$U_4^2$};  
   \node[wleaf]  at (2.5,0) {}; 
   \node[wleaf]  at (5.5,0) {};%
   \node[wleaf]  at (6.5,0) {}; %
   \node[wleaf]  at (7.25,0) {}; %
   \draw (4,1)--(2.9,-.1)--(5.1,-.1)--cycle;    
   \node at (4,0.35) {$U_4^1$};   
   \node[invertex]  at (1.25,1.25) {};  
   \node[invertex]  at (2.5,2.5) {}; 
   \node[rootvertex]  at (3.625,3.625) {}; %
   \node[invertex]  at (3.25,3.25) {}; %
   \node[invertex]  at (6,0.5) {}; %
\node at (4.5,-1) {$U_8^{1}$};  
\end{tikzpicture}\quad
\begin{tikzpicture}[scale=.5,font=\tiny,baseline={(0,0)}] 
   \draw (1,1) -- (3.625,3.625);
   \draw (1.25,1.25) -- (2.5,0);
   \draw (2.5,2.5) -- (4,1);
   \draw (2.875,2.875) -- (5.75,0);
   \draw (3.625,3.625) -- (7.25,0);
   \draw (6.25,0) -- (6.75,0.5);
   \draw (1,1)--(-.1,-.1)--(2.1,-.1)--cycle;
   \draw (4,1)--(2.9,-.1)--(5.1,-.1)--cycle;   
   \node at (1,.35) {$U_4^2$};  
   \node[wleaf]  at (2.5,0) {}; 
   \node[wleaf]  at (5.75,0) {}; 
   \node[wleaf]  at (6.25,0) {}; 
   \node[wleaf]  at (7.25,0) {}; 
   \node at (4,0.35) {$U_4^1$};   
   \node[invertex]  at (1.25,1.25) {};  
   \node[invertex]  at (2.5,2.5) {}; 
   \node[invertex]  at (2.875,2.875) {}; 
   \node[rootvertex]  at (3.625,3.625) {}; 
   \node[invertex]  at (6.75,0.5) {}; 
\node at (4.5,-1) {$U_8^{2}$};  
\end{tikzpicture}
\quad
\begin{tikzpicture}[scale=.5,font=\tiny,baseline={(0,0)}] 
   \draw (3,3)--(3.5,3.5)--(7,0);
   \draw (5.5,0.5) -- (6,0);
   \draw (2.25,2.25) -- (3,3) -- (5.5,0.5);
   \draw (1,1) -- (2.25,2.25) -- (3.5,1);
   \draw (5,0) -- (5.5,0.5);
   \draw (1,1)--(-.1,-.1)--(2.1,-.1)--cycle;
   \draw (3.5,1) -- (2.4,-.1)-- (4.6,-.1)--cycle;
   \node at (1,0.35) {$U_5$};   
   \node[wleaf]  at (5,0) {}; 
   \node[wleaf]  at (6,0) {}; 
   \node[wleaf]  at (7,0) {}; 
   \node at (3.5,0.35) {$U_4^1$};  
   \node[invertex]  at (2.25,2.25) {};
   \node[invertex]  at (3,3) {}; 
   \node[rootvertex]  at (3.5,3.5) {}; 
   \node[invertex]  at (5.5,0.5) {}; 
\node at (4,-1) {$U_8^{3}$};  
\end{tikzpicture}
\quad
\begin{tikzpicture}[scale=.5,font=\tiny,baseline={(0,0)}] 
   \draw(2.25,2.25)--(3.5,1);
   \draw(1,1)--(2.7,2.7);
   \draw(2.7,2.7)--(5.25,0);
   \draw (2.75,2.75) -- (3.5,3.5) -- (7,0);
   \draw (6,0) -- (6.5,0.5);
   \draw (1,1)--(-.1,-.1)--(2.1,-.1)--cycle;
   \draw (3.5,1) -- (2.4,-.1)-- (4.6,-.1)--cycle;
   \node at (1,0.35) {$U_5$};   
   \node[wleaf]  at (5.25,0) {}; 
   \node[wleaf]  at (6,0) {}; 
   \node[wleaf]  at (7,0) {}; 
   \node at (3.5,0.35) {$U_4^1$};  
   \node[invertex]  at (2.25,2.25) {};
   \node[invertex]  at (2.7,2.7) {}; 
   \node[rootvertex]  at (3.5,3.5) {}; 
   \node[invertex]  at (6.5,0.5) {}; 
\node at (3.5,-1) {$U_8^{4}$};  
\end{tikzpicture}\quad
\begin{tikzpicture}[scale=.5,font=\tiny,baseline={(0,0)}] 
    \draw (3,3)--(3.5,3.5)--(7,0);
   \draw (5.5,0.5) -- (6,0);
   \draw (2.25,2.25) -- (3,3) -- (5.5,0.5);
   \draw (1,1) -- (2.25,2.25) -- (3.5,1);
   \draw (5,0) -- (5.5,0.5);
   \draw (1,1)--(-.1,-.1)--(2.1,-.1)--cycle;
   \draw (3.5,1) -- (2.4,-.1)-- (4.6,-.1)--cycle;
   \node at (1,0.35) {$U_5$};   
   \node[wleaf]  at (5,0) {}; 
   \node[wleaf]  at (6,0) {}; 
   \node[wleaf]  at (7,0) {}; 
   \node at (3.5,0.35) {$U_4^2$};  
   \node[invertex]  at (2.25,2.25) {};
   \node[invertex]  at (3,3) {}; 
   \node[rootvertex]  at (3.5,3.5) {}; 
   \node[invertex]  at (5.5,0.5) {}; 
\node at (3.5,-1) {$U_8^{5}$};  
\end{tikzpicture}
\quad
\begin{tikzpicture}[scale=.5,font=\tiny,baseline={(0,0)}] 
  \draw(2.25,2.25)--(3.5,1);
   \draw(1,1)--(2.7,2.7);
   \draw(2.7,2.7)--(5.25,0);
   \draw (2.75,2.75) -- (3.5,3.5) -- (7,0);
   \draw (6,0) -- (6.5,0.5);
   \draw (1,1)--(-.1,-.1)--(2.1,-.1)--cycle;
   \draw (3.5,1) -- (2.4,-.1)-- (4.6,-.1)--cycle;
   \node at (1,0.35) {$U_5$};   
   \node[wleaf]  at (5.25,0) {}; 
   \node[wleaf]  at (6,0) {}; 
   \node[wleaf]  at (7,0) {}; 
   \node at (3.5,0.35) {$U_4^2$};  
   \node[invertex]  at (2.25,2.25) {};
   \node[invertex]  at (2.7,2.7) {}; 
   \node[rootvertex]  at (3.5,3.5) {}; 
   \node[invertex]  at (6.5,0.5) {}; 
\node at (3.5,-1) {$U_8^{6}$}; 
\end{tikzpicture}
\quad
\begin{tikzpicture}[scale=.5,font=\tiny,baseline={(0,0)}] 
   \draw(0.9,-0.1)--(3,-0.1)--(2,1)--cycle;
   \node at (2,0.3) {$C_4$}; 
   \draw(4.9,-0.1)--(7.1,-0.1)--(6,1)--cycle;
   \node at (6,0.3) {$U_4^2$}; 
   \draw(2,1)--(4.75,3.75)--(8,0.5);
   \draw(2.75,1.75)--(4,0.5);
   \draw(3.5,0)--(4,0.5)--(4.5,0);
   \draw (3.983825,2.983825)--(6,1);
   \draw(7.5,0)--(8,0.5)--(8.5,0);
   \draw(4.75,3.75)--(5.125,4.125)--(9.25,0);
   \node[wleaf]  at (3.5,0) {}; 
   \node[wleaf]  at (4.5,0) {};
   \node[wleaf]  at (7.5,0) {};
   \node[wleaf]  at (8.5,0) {};
   \node[wleaf]  at (9.25,0) {};
   \node[invertex]  at (2.75,1.75) {};
   \node[invertex]  at (4,0.5) {};
   \node[invertex]  at (3.983825,2.983825) {};
   \node[invertex]  at (4.75,3.75) {};
    \node[rootvertex]  at (5.125,4.125) {};
   \node[invertex]  at (8,0.5) {};
   \node at (5.5,-1) {\text{ $U_8^{7}$ }};
\end{tikzpicture}
\quad
\begin{tikzpicture}[scale=.5,font=\tiny,baseline={(0,0)}] 
   \draw(0.9,-0.1)--(3,-0.1)--(2,1)--cycle;
   \node at (2,0.3) {$C_4$}; 
   \draw(4.9,-0.1)--(7.1,-0.1)--(6,1)--cycle;
   \node at (6,0.3) {$U_4^2$}; 
   \draw(2,1)--(4.375,3.375)--(7.75,0);
   \draw(2.75,1.75)--(4,0.5);
   \draw(3.5,0)--(4,0.5)--(4.5,0);
   \draw (3.983825,2.983825)--(6,1);
   \draw(8.25,0)--(8.75,0.5)--(9.25,0);
   \draw(4.375,3.375)--(5.125,4.125)--(8.75,0.5);
   \node[wleaf]  at (3.5,0) {}; 
   \node[wleaf]  at (4.5,0) {};
   \node[wleaf]  at (7.75,0) {};
   \node[wleaf]  at (8.25,0) {};
   \node[wleaf]  at (9.25,0) {};
   \node[invertex]  at (2.75,1.75) {};
   \node[invertex]  at (4,0.5) {};
   \node[invertex]  at (3.983825,2.983825) {};
   \node[invertex]  at (4.375,3.375) {};
    \node[rootvertex]  at (5.125,4.125) {};
   \node[invertex]  at (8.75,0.5) {};
   \node at (5.5,-1) {\text{ $U_8^{8}$ }};
\end{tikzpicture}
\caption{The minimum size $8$-universal trees.}
\end{figure}

\begin{figure}[H]
\centering
\begin{tikzpicture}[scale=.5,font=\tiny,baseline={(0,0)}]  
   \draw(2,1)--(5.875,4.875)--(10.75,0);
   \draw(2.375,1.375)--(3.75,0);
   \draw(3.75,2.75)--(5.5,1);
   \draw(4.125,3.125)--(7.25,0);
   \draw(5.5,4.5)--(9,1);
   \draw(0.9,-0.1)--(3.1,-0.1)--(2,1)--cycle;
   \node at (2,0.3) {$U_4^2$}; 
   \draw(4.4,-0.1)--(6.6,-0.1)--(5.5,1)--cycle;
   \node at (5.5,0.3) {$U_4^1$}; 
   \draw(7.9,-0.1)--(10.1,-0.1)--(9,1)--cycle;
   \node at (9,0.3) {$C_3$}; 
   \node[wleaf]  at (3.75,0) {}; 
   \node[wleaf]  at (7.25,0) {}; 
   \node[wleaf]  at (10.75,0) {}; 
   \node[invertex]  at (2.375,1.375) {};
   \node[invertex]  at (3.75,2.75) {};
   \node[invertex]  at (4.125,3.125) {};
   \node[invertex]  at (5.5,4.5) {};
   \node[rootvertex]  at (5.875,4.875) {};
\node at (5.875,-1) {\text{ $U_9^{1}$ }};  \end{tikzpicture}
\quad
\begin{tikzpicture}[scale=.5,font=\tiny,baseline={(0,0)}] 
   \draw(2,1)--(5.5,4.5)--(10,0);
   \draw(3.375,2.375)--(4.75,1);
   \draw(3.75,2.75)--(6.5,0);
   \draw(5.125,4.125)--(8.25,1);
   \draw(0.9,-0.1)--(3.1,-0.1)--(2,1)--cycle;
   \node at (2,0.3) {$U_5$};
   \draw(3.65,-0.1)--(5.85,-0.1)--(4.75,1)--cycle;
   \node at (4.75,0.3) {$U_4^1$};
   \draw(7.15,-0.1)--(9.35,-0.1)--(8.25,1)--cycle;
   \node at (8.25,0.3) {$C_3$};
   \node[wleaf]  at (6.5,0) {}; 
   \node[wleaf]  at (10,0) {};  
   \node[invertex]  at (3.375,2.375) {};
   \node[invertex]  at (3.75,2.75) {};    
   \node[invertex]  at (5.125,4.125) {};
   \node[rootvertex]  at (5.5,4.5) {};
\node at (5.5,-1) {\text{ $U_9^{2}$ }};  \end{tikzpicture}
\quad
\begin{tikzpicture}[scale=.5,font=\tiny,baseline={(0,0)}] 
   \draw(2,1)--(5.5,4.5)--(10,0);
   \draw(3.375,2.375)--(4.75,1);
   \draw(3.75,2.75)--(6.5,0);
   \draw(5.125,4.125)--(8.25,1);
   \draw(0.9,-0.1)--(3.1,-0.1)--(2,1)--cycle;
   \node at (2,0.3) {$U_5$};
   \draw(3.65,-0.1)--(5.85,-0.1)--(4.75,1)--cycle;
   \node at (4.75,0.3) {$U_4^2$};
   \draw(7.15,-0.1)--(9.35,-0.1)--(8.25,1)--cycle;
   \node at (8.25,0.3) {$C_3$};
   \node[wleaf]  at (6.5,0) {}; 
   \node[wleaf]  at (10,0) {};  
   \node[invertex]  at (3.375,2.375) {};
   \node[invertex]  at (3.75,2.75) {};    
   \node[invertex]  at (5.125,4.125) {};
   \node[rootvertex]  at (5.5,4.5) {};
\node at (5.5,-1) {\text{ $U_9^{3}$ }};  \end{tikzpicture}
\quad
\begin{tikzpicture}[scale=.5,font=\tiny,baseline={(0,0)}] 
   \draw(2,1)--(6.375,5.375)--(11.75,0);
   \draw(2.875,1.875)--(4.25,0.5);
   \draw(3.75,0)--(4.25,0.5)--(4.75,0);
   \draw(4.25,3.25)--(6.5,1);
   \draw(4.625,3.625)--(8.25,0);
   \draw(6,5)--(10,1);
   \draw(0.9,-0.1)--(3.1,-0.1)--(2,1)--cycle;
   \node at (2,0.3) {$C_4$};
   \draw(5.4,-0.1)--(7.6,-0.1)--(6.5,1)--cycle;
   \node at (6.5,0.3) {$U_4^2$};
   \draw(8.95,-0.1)--(11.15,-0.1)--(10,1)--cycle;
   \node at (10,0.3) {$C_3$};
   \node[wleaf]  at (3.75,0) {}; 
   \node[wleaf]  at (4.75,0) {}; 
   \node[wleaf]  at (8.25,0) {}; 
   \node[wleaf]  at (11.75,0) {}; 
   \node[invertex]  at (2.875,1.875) {};
   \node[invertex]  at (4.25,0.5) {};
   \node[invertex]  at (4.25,3.25) {};
   \node[invertex]  at (4.625,3.625) {};
   \node[invertex]  at (6,5) {};
   \node[rootvertex]  at (6.375,5.375) {};
\node at (6.375,-1) {\text{ $U_9^{4}$ }};  \end{tikzpicture}
\quad
\begin{tikzpicture}[scale=.5,font=\tiny,baseline={(0,0)}] 
   \draw(2,1)--(5.5,4.5)--(10,0);
   \draw(3.375,2.375)--(4.75,1);
   \draw(3.75,2.75)--(6.5,0);
   \draw(5.125,4.125)--(8.25,1);
   \draw(0.9,-0.1)--(3.1,-0.1)--(2,1)--cycle;
   \node at (2,0.3) {$U_5$};
   \draw(3.65,-0.1)--(5.85,-0.1)--(4.75,1)--cycle;
   \node at (4.75,0.3) {$C_4$};
   \draw(7.15,-0.1)--(9.35,-0.1)--(8.25,1)--cycle;
   \node at (8.25,0.3) {$B_2$};
   \node[wleaf]  at (6.5,0) {}; 
   \node[wleaf]  at (10,0) {};  
   \node[invertex]  at (3.375,2.375) {};
   \node[invertex]  at (3.75,2.75) {};    
   \node[invertex]  at (5.125,4.125) {};
   \node[rootvertex]  at (5.5,4.5) {};
\node at (5.5,-1) {\text{ $U_9^{5}$ }};  \end{tikzpicture}
\quad
\begin{tikzpicture}[scale=.5,font=\tiny,baseline={(0,0)}] 
   \draw(2,1)--(5.5,4.5)--(10,0);
   \draw(3.375,2.375)--(4.75,1);
   \draw(3.75,2.75)--(6.5,0);
   \draw(5.125,4.125)--(8.25,1);
   \draw(0.9,-0.1)--(3.1,-0.1)--(2,1)--cycle;
   \node at (2,0.3) {$U_5$};
   \draw(3.65,-0.1)--(5.85,-0.1)--(4.75,1)--cycle;
   \node at (4.75,0.3) {$B_2$};
   \draw(7.15,-0.1)--(9.35,-0.1)--(8.25,1)--cycle;
   \node at (8.25,0.3) {$C_4$};
   \node[wleaf]  at (6.5,0) {}; 
   \node[wleaf]  at (10,0) {};  
   \node[invertex]  at (3.375,2.375) {};
   \node[invertex]  at (3.75,2.75) {};    
   \node[invertex]  at (5.125,4.125) {};
   \node[rootvertex]  at (5.5,4.5) {};
\node at (5.5,-1) {\text{ $U_9^{6}$ }};  \end{tikzpicture}
\quad
\begin{tikzpicture}[scale=.5,font=\tiny] 
\draw (1,1) -- (2.25,2.25) -- (3.5,1);
 \draw (2.25,2.25) -- (2.75,2.75) -- (5.5,0);
 \draw (1,1)--(-0.1,-.1)--(2.1,-.1)--cycle;
 \node at (1,0.35) {$U_7$};  
 \draw(2.4,-.1)--(3.5,1)--(4.6,-.1)--cycle;
 \node at (3.5,0.35) {$U_4^1$};  
 \node[wleaf]  at (5.5,0) {};
 \node[invertex]  at (2.25,2.25) {};
 \node[rootvertex]  at (2.75,2.75) {};
\node at (2.75,-1) {$U_9^7$}; 
\end{tikzpicture}
\quad
\begin{tikzpicture}[scale=.5,font=\tiny] 
\draw (1,1) -- (2.25,2.25) -- (3.5,1);
 \draw (2.25,2.25) -- (2.75,2.75) -- (5.5,0);
 \draw (1,1)--(-0.1,-.1)--(2.1,-.1)--cycle;
 \node at (1,0.35) {$U_7$};  
 \draw(2.4,-.1)--(3.5,1)--(4.6,-.1)--cycle;
 \node at (3.5,0.35) {$U_4^2$};  
 \node[wleaf]  at (5.5,0) {};
 \node[invertex]  at (2.25,2.25) {};
 \node[rootvertex]  at (2.75,2.75) {};
\node at (2.75,-1) {$U_9^8$}; 
\end{tikzpicture}
\caption{The minimum size $9$-universal trees.}
\end{figure}
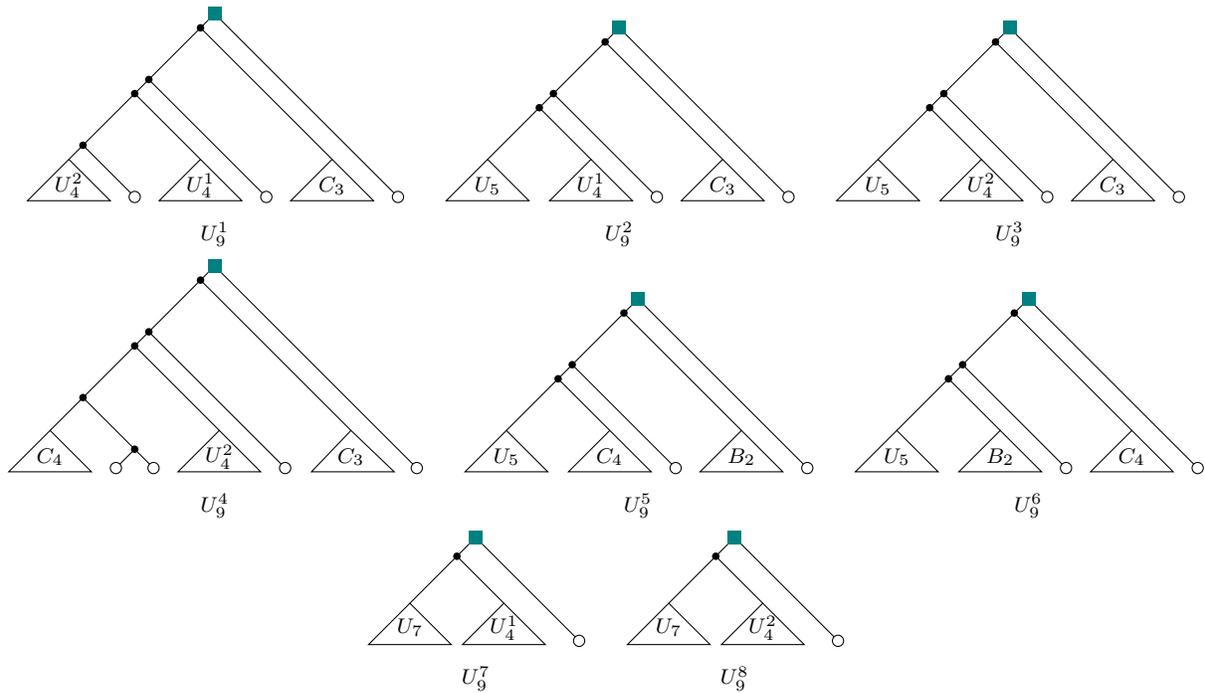

\begin{figure}[H]
   \centering
 \begin{tikzpicture}[scale=.5,font=\tiny,baseline={(0,0)}] 
   \draw(2.25,2.25)--(3.5,1);
   \draw(1,1)--(2.7,2.7);
   \draw(2.7,2.7)--(5.25,0);
   \draw (2.75,2.75) -- (3.5,3.5) -- (7,0);
   \draw (6,0) -- (6.5,0.5);
   \draw (1,1)--(-.1,-.1)--(2.1,-.1)--cycle;
   \draw (3.5,1) -- (2.4,-.1)-- (4.6,-.1)--cycle;
   \node at (1,0.35) {$U_7$};   
   \node[wleaf]  at (5.25,0) {}; 
   \node[wleaf]  at (6,0) {}; 
   \node[wleaf]  at (7,0) {}; 
   \node at (3.5,0.35) {$U_5$};  
   \node[invertex]  at (2.25,2.25) {};
   \node[invertex]  at (2.7,2.7) {}; 
   \node[rootvertex]  at (3.5,3.5) {}; 
   \node[invertex]  at (6.5,0.5) {}; 
\node at (3.5,-1) {$U_{10}^{1}$};  
\end{tikzpicture}\quad
\begin{tikzpicture}[scale=.5,font=\tiny,baseline={(0,0)}] 
    \draw (3,3)--(3.5,3.5)--(7,0);
   \draw (5.5,0.5) -- (6,0);
   \draw (2.25,2.25) -- (3,3) -- (5.5,0.5);
   \draw (1,1) -- (2.25,2.25) -- (3.5,1);
   \draw (5,0) -- (5.5,0.5);
   \draw (1,1)--(-.1,-.1)--(2.1,-.1)--cycle;
   \draw (3.5,1) -- (2.4,-.1)-- (4.6,-.1)--cycle;
   \node at (1,0.35) {$U_7$};   
   \node[wleaf]  at (5,0) {}; 
   \node[wleaf]  at (6,0) {}; 
   \node[wleaf]  at (7,0) {}; 
   \node at (3.5,0.35) {$U_5$};  
   \node[invertex]  at (2.25,2.25) {};
   \node[invertex]  at (3,3) {}; 
   \node[rootvertex]  at (3.5,3.5) {}; 
   \node[invertex]  at (5.5,0.5) {}; 
\node at (3.5,-1) {$U_{10}^{2}$};  
\end{tikzpicture}
\quad
       \begin{tikzpicture}[scale=.5,font=\tiny,baseline={(0,0)}] 
       \draw(1,1)--(2.25,2.25)--(3.5,1);
       \draw(2.25,2.25)--(2.675,2.675)--(5.25,0);
       \draw(2.675,2.675)--(4,4)--(7,1);
       \draw(4,4)--(4.375,4.375)--(8.75,0);
       \draw(-.1,-.1)--(2.1,-.1)--(1,1)--cycle;
       \draw(2.4,-.1)--(4.6,-.1)--(3.5,1)--cycle;
       \draw(5.9,-.1)--(8.1,-.1)--(7,1)--cycle;
       \node at (1.1,0.35) {$U_7$}; 
       \node at (3.65,0.35) {$U_5$}; 
       \node at (7,0.35) {$C_3$}; 
       \node[wleaf]  at (5.25,0) {}; 
       \node[wleaf]  at (8.75,0) {}; 
       \node[invertex]  at (2.25,2.25) {};
       \node[invertex]  at (2.675,2.675) {};
       \node[invertex]  at (4,4) {};
       \node[rootvertex]  at (4.375,4.375) {};
       \node at (5,-1) {\text{ $U_{11}$ }};  
       \end{tikzpicture}
   \caption{The minimum size $n$-universal trees for $n\in\{10,11\}$.}
\end{figure}
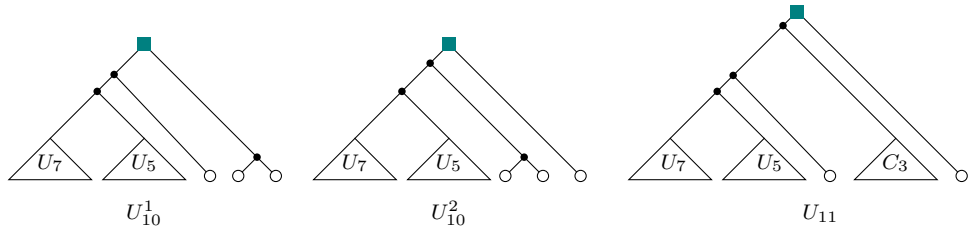

We conclude this section with the following open problems:
\begin{problem}
Is it true, that for every $n$, there exists an $n$-universal tree of minimum size, in which every leaf has distance at
most $n-1$ from the root?
\end{problem}
The computational results provided such $n$-universal trees for $n\leq 11$. However, our recursive upper bound for the universal tree size results in trees that exceed this distance, and certainly are not optimal.
 
\begin{problem}
Find (a good conjecture for) a general construction for an $n$-universal tree of minimum size for all $n$.
\end{problem}

\section{Tanglegrams} \label{sec:kovetk}
In this section, we turn our attention to tanglegrams, which are essential tools in the study of co-evolution (see, e.g.,\cite{Burt}, \cite{HafnerNadler}, \cite{Matsen}, \cite{page}). We start with some basic definitions. 

\begin{definition} A  {\em (binary) tanglegram} $\T=({L},{R},\sigma)$ is a graph that consists of a left  rooted binary tree ${L}$, a right  rooted 
binary tree ${R}$, such that $\whiteleaf{R}=\whiteleaf{L}$, and a perfect matching  $\sigma$ between the leaves of ${L}$ and ${R}$ (see Figure~\ref{fig:tanglegram}). 
\end{definition} 
Two tanglegrams are considered {\it isomorphic}, if there is a graph isomorphism between
them fixing $r_L$ and $t_R$. The {\em size} of a tanglegram is $\whiteleaf{R}=\whiteleaf{L}=|\sigma|$.
S.C.~Billey, M.~Konvalinka, and F.A.~Matsen \cite{billey} enumerated tanglegrams of size $n$ up to isomorphism, and obtained that their number is asymptotically
\begin{equation} \label{tanglecount}
n! \frac{e^{1/8}4^{n-1}}{\pi n^3}.
\end{equation}

\begin{definition} Given a tanglegram $\T=({L},{R},\sigma)$ and a subset $\emptyset\not= \sigma'\subseteq \sigma$
of matching edges,  the {\it subtanglegram of $\T$ induced by $\sigma'$},  $({L'},{R'},\sigma')$ is defined as follows. The endpoints of the edges in
$\sigma'$ identify a leaf set $X$ in $L$ and a leaf set $Y$ in $R$. $X$ defines an induced binary subtree $L'$ of $L$ 
and $Y$ defines an induced binary subtree $R'$ of $R$. Keep the edges of $\sigma'$ between the leaves 
of $L'$ and $R'$ (see Figure~\ref{fig:tanglegram}). 
\end{definition}

\begin{definition} A tanglegram $\T_1$ is an {\it induced subtanglegram} of the tanglegram $\T$, if $\T_1$ is isomorphic
to a subtanglegram of $\T$ induced by some subset of matching edges.
\end{definition}

\begin{definition} A tanglegram $\T$ is {\it $n$-universal}, if it contains all tanglegrams of size $n$ as induced 
subtanglegrams.
\end{definition}

\begin{figure}[htbp]
\centering
\begin{tikzpicture}[scale=0.65,font=\tiny]
       \draw (2,-2)--(0,0)--(2,2);
       \draw (2,1)--(1.5,1.5);
       \draw (2,0)--(1,-1);
       \draw (2,-1)--(1.5,-1.5);
        \draw[dashed] (2,2)--(3,2);
       \draw[dashed] (2,1)--(3,1);
       \draw[dashed] (2,0)--(3,0);
       \draw[dashed] (2,-1)--(3,-1);
       \draw[dashed] (2,-2)--(3,-2);
      
       \node[invertex]  at (1,-1) {};
       \node[invertex]  at (1.5,1.5) {};
       \node[invertex]  at (1.5,-1.5) {};
       \node[wleaf]  at (2,2) {};
       \node[wleaf]  at (2,1) {};
       \node[wleaf]  at (2,0) {};
       \node[wleaf]  at (2,-1) {};
       \node[wleaf]  at (2,-2) {};
       \node[rootvertex]  at (0,0) {};
       
       \draw (3,-2)--(5,0)--(3,2);
        \draw (3,1)--(3.5,1.5);
       \draw (3,0)--(4,1);
       \draw (3,-1)--(3.5,-1.5);
        \node[wleaf]  at (3,2) {};
       \node[wleaf]  at (3,1) {};
       \node[wleaf]  at (3,0) {};
       \node[wleaf]  at (3,-1) {};
       \node[wleaf]  at (3,-2) {};
       \node[invertex]  at (3.5,1.5) {};
       \node[invertex]  at (3.5,-1.5) {};
       \node[invertex]  at (4,1) {};
       \node[rootvertex]  at (5,0) {};
	   \node at (2.5,2.2) {$e_1$};
    \node at (2.5,1.2) {$e_2$};    
 	\node at (2.5,.2) {$e_3$};

      \draw (7.5,1)--(6.5,0)--(7.5,-1);
      \draw (7.5,0)--(7,.5);
       \draw[dashed] (7.6,1)--(8.5,1);
       \draw[dashed] (7.5,0)--(8.5,0);
       \draw[dashed] (7.5,-1)--(8.5,-1);
       \draw (8.5,1)--(9.5,0)--(8.5,-1);
       \draw (8.5,0)--(9,.5);
        \node[rootvertex]  at (6.5,0) {};      
       \node[invertex]  at (7,.5) {};
      \node[wleaf]  at (7.5,1) {};
      \node[wleaf]  at (7.5,0) {};
      \node[wleaf]  at (7.5,-1) {};
       \node[invertex]  at (9,.5) {};	
       \node[wleaf]  at (8.5,1) {};
      \node[wleaf]  at (8.5,0) {};
      \node[wleaf]  at (8.5,-1) {};
 	\node at (8,1.2) {$e_1$};
   	\node at (8,.2) {$e_2$};    
 	\node at (8,.-.8) {$e_3$};
  \node[rootvertex]  at (9.5,0) {};

\end{tikzpicture}
\caption{A binary tanglegram $\T=(L,R,\sigma)$ with matching edges $e_1,e_2,e_3$ selected and the subtanglegram induced by the selected edges.}\label{fig:tanglegram}
\end{figure}
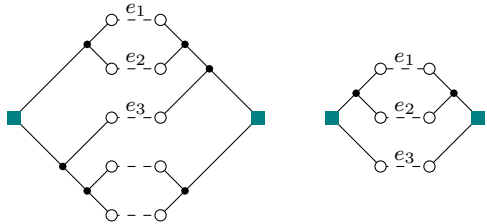

We now obtain upper and lower bounds for the minimum size of $n$-universal tanglegrams. 
Let $u(n)$  denote  the minimum size of an $n$-universal tree, and let  $q(n)$  denote  the minimum size of an $n$-universal 
tanglegram. 

\begin{theorem}\label{th:untang}
 The following holds:
\begin{equation}
q(n) \leq u(n)^2= O(n^4).
\end{equation}
\end{theorem} 
\begin{proof}
Take two copies, $L$ and $R$, of the minimum size $n$-universal tree. Each has $u(n)$ leaves.
Both in $L$ and $R$, attach to every leaf an {\it arbitrary} $u(n)$-leaf rooted binary tree, and obtain in this way trees $L'$ and $R'$, each with
$u(n)^2$ leaves. Both $L'$ and $R'$ have $u(n)$ ``small trees" rooted in the leaves of the starting trees $L$ and $R$. 
Now construct a tanglegram $\T$ with trees $L'$ and $R'$ in the following way: for every small tree in $L'$, send a matching 
edge from one of the $u(n)$ leaves of the small tree into exactly one of the $u(n)$ small trees of $R'$ (i.e., ``match" each of the $u(n)$ leaves of a small tree in $L'$ with a distinct small tree of $R'$). Now the tanglegram $\T$ is $n$-universal. Indeed, consider any tanglegram $\T_1$ of size $n$ with left tree $L_1$ and right tree $R_1$.
By the choice of the universal tree, $L_1$ is an induced binary subtree of $L$ and $R_1$ is an induced binary subtree of $R$. Now, whatever the matching between the left tree $L_1$ and right tree $R_1$ is in the tanglegram, we inserted edges 
between the leaves of  $L'$ and $R'$ in the construction of $\T$ that realize this matching. The claim now follows from Theorem~\ref{thm:upperbound}.
\end{proof}

\begin{theorem}
\begin{equation}
q(n)\geq  \frac{4n^2}{e^2}(1+o(1)).
\end{equation}
\end{theorem} 
\begin{proof}
Comparing all possible size $n$ subtanglegrams of the universal tanglegram to the number of size $n$ tanglegrams in \eqref{tanglecount}, we have 
$$\frac{q(n)^n}{n!}\geq \binom{q(n)}{n}\geq  n! \frac{e^{1/8}4^{n-1}}{\pi n^3}(1+o(1)),$$
which, using Stirling's formula, gives the claimed result. 
\end{proof}

\section{Generalization for \texorpdfstring{$d$}{d}-ary trees and \texorpdfstring{$d$}{d}-ary tanglegrams}\label{sec:dary}
We conclude this manuscript with a brief section on $d$-ary trees and $d$-ary tanglegrams.

\begin{definition} Let $d\ge 2$.  A \emph{rooted $d$-ary tree} is a rooted tree where every non-leaf vertex has at least two and at most $d$ descendants. The 1-vertex tree is a rooted $d$-ary tree.
An \emph{r-$d$-r tree} is a rooted $d$-ary redleaf tree.
\end{definition}

\begin{definition} Let $T$ be a rooted $d$-ary tree (redleaf or not) and let $S$ be a subset of its leaves. The \emph{$d$-ary subtree induced by $S$} 
is a rooted $d$-ary tree $T^{\star}$ obtained as follows:
Let $F$ be the smallest subtree  of $T$ that contains $S$, and let $r_{T^{\star}}$ be the vertex of $F$ closest to $r_T$. $F$ is a subdivision of a unique rooted $d$-ary tree $T^{\star}$. In other words, we obtain $T^{\star}$ by suppressing the non-root  vertices of degree 2 in $F$. The leaves maintain their original coloration, i.e., $T^{\star}$ is a redleaf tree precisely when $T$ is a redleaf tree and $S$ contains $\ell_T$. The tree $T^{\prime}$ is an \emph{induced $d$-ary subtree} of $T$ if it is induced by some subset $S^{\prime}$ of the leaves of $T$. 
\end{definition}

\begin{definition} A rooted $d$-ary tree is \emph{$n$-universal} if it contains all $n$-leaf rooted $d$-ary trees as induced $d$-ary subtrees (see Figure~\ref{fig:daryuniversal}). 
$\mathcal{U}^{(d)}(n)$ denotes the class of $n$-universal $d$-ary trees and $u^{(d)}(n)=\min\{\whiteleaf{T}:T\in\mathcal{U}^{(d)}(n)\}$.
An r-$d$-r tree is \emph{$n$-universal} if it contains all r-$d$-r trees $T$ with $n$ white leaves as induced $d$-ary subtrees. $\mathcal{U}_{1}^{(d)}(n)$ denotes the class of $n$-universal r-$2$-r trees and $u_1^{(d)}(n)=\min\{\whiteleaf{T}:T\in\mathcal{U}_1^{(d)}(n)\}$.
\end{definition}

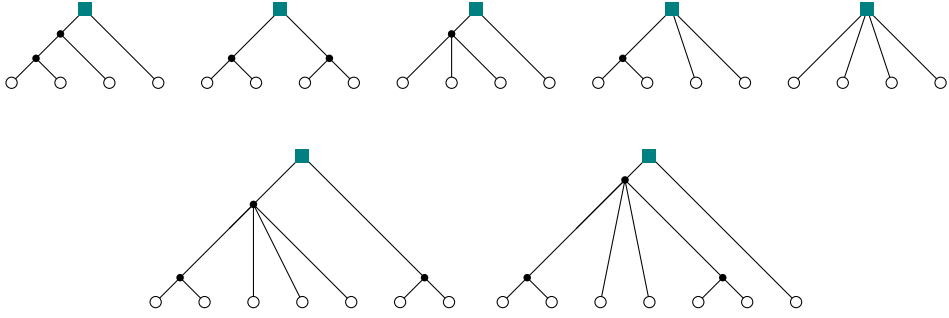
\begin{figure}[htbp]
   \centering
    \begin{tikzpicture}[scale=.65,font=\tiny] 
    \draw(1,6)--(2.5,7.5)--(4,6);
    \draw(2,6)--(1.5,6.5);
    \draw(3,6)--(2,7);
    \node[wleaf]  at (1,6) {}; 
    \node[wleaf]  at (2,6) {}; 
    \node[wleaf]  at (3,6) {}; 
    \node[wleaf]  at (4,6) {}; 
    \node[invertex]  at (1.5,6.5) {};
    \node[invertex]  at (2,7) {};
    \node[rootvertex]  at (2.5,7.5) {};

    \draw(5,6)--(6.5,7.5)--(8,6);
    \draw(6,6)--(5.5,6.5);
    \draw(7,6)--(7.5,6.5);
    \node[wleaf]  at (5,6) {}; 
    \node[wleaf]  at (6,6) {}; 
    \node[wleaf]  at (7,6) {}; 
    \node[wleaf]  at (8,6) {}; 
    \node[invertex]  at (5.5,6.5) {};
    \node[invertex]  at (7.5,6.5) {};
    \node[rootvertex]  at (6.5,7.5) {};

    \draw(9,6)--(10.5,7.5)--(12,6);
    \draw(10,6)--(10,7)--(11,6);
    \node[wleaf]  at (9,6) {}; 
    \node[wleaf]  at (10,6) {}; 
    \node[wleaf]  at (11,6) {}; 
    \node[wleaf]  at (12,6) {}; 
    \node[invertex]  at (10,7) {};
    \node[rootvertex]  at (10.5,7.5) {};

    \draw(13,6)--(14.5,7.5)--(16,6);
    \draw(14,6)--(13.5,6.5);
    \draw(15,6)--(14.5,7.5);
    \node[wleaf]  at (13,6) {}; 
    \node[wleaf]  at (14,6) {}; 
    \node[wleaf]  at (15,6) {}; 
    \node[wleaf]  at (16,6) {}; 
    \node[invertex]  at (13.5,6.5) {};
    \node[rootvertex]  at (14.5,7.5) {};

    \draw(17,6)--(18.5,7.5)--(20,6);
    \draw(18,6)--(18.5,7.5);
    \draw(19,6)--(18.5,7.5);
    \node[wleaf]  at (17,6) {}; 
    \node[wleaf]  at (18,6) {}; 
    \node[wleaf]  at (19,6) {}; 
    \node[wleaf]  at (20,6) {}; 
    \node[rootvertex]  at (18.5,7.5) {};
    \end{tikzpicture}\vskip .3in
    
    \begin{tikzpicture}[scale=.65,font=\tiny] 
     \draw (0,0)--(.5,.5)--(1,0);
     \draw (.5,.5)--(2,2)--(2,0);
     \draw (3,0)--(2,2)--(4,0);
     \draw (1.5,1.5)--(3,3)--(5.5,.5);
     \draw (5,0)--(5.5,.5)--(6,0);    
     \node[wleaf]  at (0,0) {}; 
     \node[wleaf]  at (1,0) {};   
     \node[wleaf]  at (2,0) {}; 
     \node[wleaf]  at (3,0) {};   
     \node[wleaf]  at (4,0) {}; 
     \node[wleaf]  at (5,0) {};   
      \node[wleaf]  at (6,0) {}; 
     \node[invertex]  at (.5,.5) {};  
     \node[invertex]  at (5.5,.5) {}; 
     \node[invertex]  at (2,2) {};   
     \node[rootvertex]  at (3,3) {}; 
     \end{tikzpicture}
     \quad
   \begin{tikzpicture}[scale=.65,font=\tiny] 
     \draw (0,0)--(.5,.5)--(1,0);
     \draw (.5,.5)--(2.5,2.5)--(2,0);
     \draw (3,0)--(2.5,2.5)--(4.5,.5);
     \draw (1.5,1.5)--(3,3)--(6,0);
     \draw (4,0)--(4.5,.5)--(5,0);    
     \node[wleaf]  at (0,0) {}; 
     \node[wleaf]  at (1,0) {};   
     \node[wleaf]  at (2,0) {}; 
     \node[wleaf]  at (3,0) {};   
     \node[wleaf]  at (4,0) {}; 
     \node[wleaf]  at (5,0) {};   
      \node[wleaf]  at (6,0) {}; 
     \node[invertex]  at (.5,.5) {};  
     \node[invertex]  at (4.5,.5) {}; 
     \node[invertex]  at (2.5,2.5) {};   
     \node[rootvertex]  at (3,3) {}; 
     \end{tikzpicture}
    \caption{All rooted $4$-ary trees of size four and the two $4$-universal rooted $4$-ary trees of minimum size.}
    \label{fig:daryuniversal}
\end{figure}

We also need to extend the definition of the $\oplus$ operation to more than two trees.
\begin{definition} For any $s\ge 2$, let $T_1, \ldots, T_s$ be rooted trees, at most one of which is a redleaf tree. Then, the tree $T=\oplus(T_1,\ldots,T_s)$ is obtained by joining a root vertex $r_{T}$ to the roots of each tree $T_i$ for $1 \leq i \leq s$, i.e., by introducing a new vertex $r_{T}$ and the edges $(r_T,r_{T_i})$ for $1 \leq i \leq s$. $T$ is a redleaf tree precisely when one of the trees $T_1, \ldots, T_s$ is a redleaf tree.
\end{definition}

The proof of Theorem~\ref{thm:lowerbound} extends to $d$-ary trees, so
$u^{(d)}(n)=O(n\log(n))$.
Moreover, the following upper bounds can be proven similarly to the proofs of the corresponding binary versions earlier. 

\begin{lemma} Let $n,d\ge 2$, 
$T_{d+1}\in\mathcal{U}^{(d)}_1\left(\lfloor\frac{n}{2}\rfloor\right)$, 
and for $j\in[d]$, let
$T_j\in\mathcal{U}^{(d)}\left(\lfloor\frac{n}{2}\rfloor\right)$. Then, we have $T_{d+1}[\oplus(T_1,\ldots, T_{d})]\in\mathcal{U}^{(d)}(n)$ and consequently,
$$u^{(d)}(n)\le
u_1^{(d)}\left(\left\lfloor \frac{n}{2} \right\rfloor\right) 
+d\cdot u^{(d)}\left(\left\lfloor \frac{n}{2} \right\rfloor\right) .$$
\end{lemma}

\begin{lemma} Let $n,d\ge 2$,
$T_{d},T_{d+1}\in\mathcal{U}_1^{(d)}\left(\left\lfloor\frac{n}{2}\right\rfloor\right)$,  $T_{d-1}\in\mathcal{U}^{(d)}(n)$, and for $j\in[d-2]$, let
 $T_j\in\mathcal{U}^{(d)}\left(\lfloor\frac{n}{2}\rfloor\right)$.
 Then, $T_{d+1}[\oplus (T_1,T_2,\ldots T_{d})]\in\mathcal{U}_1^{(d)}(n)$, and consequently
\begin{eqnarray*}
u_1^{(d)}(n)&\le&  2u_1^{(d)}\left(\left\lfloor \frac{n}{2}\right\rfloor \right) +u^{(d)}(n)
+(d-2)u^{(d)}\left(\left\lfloor \frac{n}{2}\right\rfloor \right)\\
&\leq&3u_1^{(d)}\left(\left\lfloor \frac{n}{2}\right\rfloor\right) 
+(2d-2)u^{(d)}\left(\left\lfloor \frac{n}{2}\right\rfloor\right). 
\end{eqnarray*}
\end{lemma}

\begin{lemma}\label{lm:dseq} Fix  $d\ge 2$.
Define the $\{A_i\}, \{B_i\},\{C_i\}$ sequences  with initial conditions
$A_1=1$, $B_1=d$, 
and recursions
 $A_{k+1}=3A_k+B_k$, $B_{k+1}=(2d-2)A_k+dB_k$  for $k\ge 1$. Moreover, for all $k\geq 1$, define $C_k=2A_k+B_k$.
Then $C_k=(d+2)^{k}$, and consequently $A_k,B_k=O((d+2)^k)$. 
\end{lemma}

\begin{lemma} 
Using $A_k,B_k$ from Lemma~\ref{lm:dseq} for $n\ge 2^k$ we have
$$
u^{(d)}(n)\le A_k u^{(d)}_1\left(\left\lfloor2^{-k} n\right\rfloor\right)+B_ku^{(d)}\left(\left\lfloor2^{-k} n\right\rfloor\right).$$
\end{lemma}

\begin{theorem} For any fixed $d\ge 2$, we have
 $u^{(d)}(n)=O(n^{\log_2(d+2)}).$
\end{theorem}

As before, we can use the sequence $\vec{s}_i$ introduced in Definition~\ref{def:svec} to connect the growth rates of $u^{(d)}(n)$ and $u_1^{(d)}(n)$.
\begin{theorem} Let $d\ge 2$. The following are true.
\begin{enumerate}[label={\upshape (\roman*)}]
\item Fix any $k\in\mathbb{N}$ and for each $i:1\le i\le 2^{k+1}-1$, let $R_i=\oplus(T_{i,1},\ldots,T_{i,d-1})$, where each $T_{i,j}$ is an $s_{i,k}$-universal tree of size $u^{(d)}(s_{i,k})$, let $Q$ be the redleaf tree with no white leaves, and set $T=R_{1}\oplus(R_{2}\oplus(R_{3}\oplus(\cdots(R_{2^{k+1}-1}\oplus Q)\cdots)))$. Then, $T\in\mathcal{U}_1^{(d)}(2^k)$.
\item For every $k\in\mathbb{N}$, $u_1^{(d)}(2^k)\le(d-1) \sum_{i=0}^k 2^i u^{(d)} (2^{k-i})$.
\item For any $\alpha>1$, we have $u^{(d)}(n)=O(n^{\alpha})$ if and only if $u_1^{(d)}(n)=O(n^{\alpha})$.
\item The bound $u_1^{(d)}(n)=O\left(n^{\log(d+2)}\right)$ holds.
\end{enumerate}
\end{theorem}

Finally, we briefly consider $d$-ary tanglegrams.
\begin{definition} A  {\em $d$-ary tanglegram} $\T=({L},{R},\sigma)$ is a graph that consists of a left  rooted $d$-ary tree ${L}$, a right  rooted 
$d$-ary tree ${R}$, such that $\whiteleaf{R}=\whiteleaf{L}$, and a perfect matching  $\sigma$ between their leaves of ${L}$ and ${R}$. 
\end{definition}
 
It is clear how to define  induced $d$-ary subtanglegrams and universal $d$-ary tanglegrams extending the ``binary" definitions. Denoting the minimum size of a universal $d$-ary tanglegram by $q^{(d)}(n)$, the proof of Theorem~\ref{th:untang} easily yields the following bound.
\begin{theorem} 
 The following holds:
\begin{equation}
q^{(d)}(n) \leq \left(u^{(d)}(n)\right)^2= O(n^{2\log_2(d+2)}).
\end{equation}
\end{theorem}


\begin{thebibliography}{10}


\bibitem{billey}
S.C.~Billey, M.~Konvalinka, and F.A.~Matsen,
\newblock On the enumeration of tanglegrams and tangled chains,
\newblock {\it J. Combinatorial Theory Ser. A}, {\bf146}(2017), 239--263.

\bibitem{bordewich} M. Bordewich, S. Linz, M. Owen, K. St. John, C. Semple, and K. Wicke, 
On the Maximum Agreement Subtree conjecture for balanced trees, {\it SIAM J.   Disc. Math.}  {\bf 36}(1)(2022)

\bibitem{Burt}
A.~Burt and R.~Trivers,
\newblock {Genes in Conflict},
\newblock Belknap Harvard Press, Cambridge MA, 2006.

\bibitem{kailai} Kai Lai Chung, On the probability of the occurrence of at least $m$ events among $n$ arbitrary events, {\it Ann. Math. Stat.}  {\bf 12}(3)(1941), 328--338.

\bibitem{cater} F. R. K. Chung, R. L. Graham, and J. Shearer, Universal caterpillars, {\it  J. Combin.
Theory} Ser. B, {\bf 31}(1981), 348--355.

\bibitem{Coppersmith} F.R.K. Chung, R.L. Graham and D. Coppersmith, On trees containing
all small tree, in: {\it The Theory of Applications of Graphs}, G. Chartrand, (Editor), John Wiley and Sons,
1981,  265--272.

\bibitem{dary} \'E. Czabarka, A.A.V. Dossou-Olory, L.A. Sz\'ekely, S. Wagner, Inducibility of $d$-ary trees , {\it Discr. Math.} {\bf 343}(2)(2020), Article 111671


\bibitem{supertrees} C. Defant, N. Kravitz, A. Shah, Supertrees, {\it Electr. J. Comb.} {\bf 27}(2)(2020), \#P2.7, 26 pages. 

\bibitem{EFRS} Paul Erd\H os, R.J. Faudree, C.C. Rousseau, and R.H. Schelp, 
The Book-Tree Ramsey Numbers, {\it Scientia}
Series A: {\it Mathematical Sciences}, Vol. {\bf 1} (1988),  111--117,
Universidad T\'echnica Federico Santa Mar\'\i a
Valpara\'\i so, Chile.

\bibitem{artificial} P\'eter L. Erd\H os,  M.A. Steel, L.A. Sz\'ekely, and T.J. Warnow,
Local quartet splits of a binary tree infer all quartet splits via one
dyadic inference rule,  {\it Computers and Artificial Intelligence}
{\bf 16}(2)(1997), 217--227. 


\bibitem{engen}
  M. Engen and  V. Vatter, Containing all permutations,
{\it Amer. Math. Monthly}, {\bf 128}(1)(2021). 

\bibitem{galambos} J. Galambos and I. Simonelli, {\it Bonferroni-Type Inequalities with Applications}, Springer-Verlag, 1996.

\bibitem{goldberg} M. Goldberg and E. Lifshitz, On minimal universal trees, {\it Mat. Zametki} {\bf 4}(1968), 371--378 (in Russian).

\bibitem{HafnerNadler}
M.S.~Hafner and S.A.~Nadler,
\newblock Phylogenetic trees support the coevolution of parasites and their
  hosts,
\newblock {Nature}, {\bf332}(1988), 258--259.


\bibitem{kalmarmagyar} L. Kalm\'ar, A "factorisatio numerorum" probl\'em\'aj\'ar\'ol, 
{\it Matematikai \'es Fizikai Lapok}, {\bf XXXVIII}(1931),  1--15 (in Hungarian).

\bibitem{kalmarnemet}  L. Kalm\'ar, \"Uber die mittlere Anzahl der Produktdarstellungen
der Zahlen, (Erste Mitteilung), {\it Acta
Litterarum ac Scientiarum
Regiae Universitatis Hungaricae Francisco-Josephinae
sectio
Scientiarum Mathematicarum} {\bf 5}(1030-1932),  95--107.

\bibitem{markin} A. Markin, On the extremal maximum agreement subtree problem, {\it Discrete Appl. Math.} {\bf 285}(2020),  612--620.

\bibitem{Matsen}
F.A.~Matsen, S.C. Billey, A.~Kas, and M.~Konvalinka,
\newblock Tanglegrams: a reduction tool for mathematical phylogenetics,
\newblock {IEEE/ACM Transactions on Computational Biology and Bioinformatics}, {\bf15}(1)(2016), 343--349.


\bibitem{otter}  R. Otter,  The number of trees.
{\it Annals of Math.}, 2nd Ser., {\bf 49}(3)(1948),  583--599.

\bibitem{Landau1977} B.V.~Landau, An asymptotic expansion for the Wedderburn-Etherington
sequence. {\it Mathematika}, {\bf 24}(2)(1977), 262--265.

\bibitem{page} R.D.M.~Page, (Editor) Tangled  Trees. Phylogeny, Cospeciation and Coevolution, University of Chicago Press, Chicago IL, 2002.

\bibitem{rado} R. Rado, Universal graphs and universal functions, {\it Acta Arithmetica},  {\bf 9}(4)(1964), 331--340.

\bibitem{SempleSteel} C.~Semple, M.~A. Steel, {\it Phylogenetics}, Oxford University Press, 2003.

\bibitem{masttree} M. Steel and T. Warnow, Kaikoura tree theorems: Computing the maximum agreement subtree, {\it Inform. Proc. Letters}
{\bf 48}(2),  77--82.

\bibitem{dtangle} A. Wotzlaw, E. Speckenmeyer, S. Porschen, Generalized $d$-ary tanglegrams on level graphs: A satisfiability-based approach and its evaluation, {\it Discr. Appl. Math.}, {\bf160}(16-17)(2012), 2349--2363.

\bibitem{OEIS} The On-Line Encyclopedia of Integer Sequences (OEIS) \url{https://oeis.org/}
\end{thebibliography}
\end{document}